\documentclass[10pt]{amsart}
\usepackage{epsf,latexsym,amsmath,amssymb,amscd,datetime}
\usepackage{amsmath,amsthm,amssymb,enumerate,eucal,url,calligra,mathrsfs}

\usepackage{blkarray} % July 5, 2017

\usepackage{tikz,scalerel}
\usetikzlibrary{arrows}

\usepackage{imakeidx}     % but this works great otherwise!

\usepackage[outdir=./]{epstopdf}% November 15, 2018

\usepackage{graphicx}
\usepackage{color}
\newenvironment{jfnote}{ \bgroup \color{blue} }{\egroup}
\newcommand{\oldStuff}[1]{}

\DeclareMathOperator{\SHom}{\mathscr{H}\text{\kern -3pt {\calligra\large om}}\,}

\IfFileExists{my_xrefs}{\input my_xrefs}{}
\newcommand{\naturals}{{\mathbb N}}

\DeclareMathOperator{\Trace}{Trace}

\usepackage{mathrsfs}
\usepackage{amssymb}
\usepackage{dsfont}
\usepackage{verbatim}
\usepackage{url}

\theoremstyle{plain}
\newtheorem{theorem}{Theorem}[section]
\newtheorem{lemma}[theorem]{Lemma}

\theoremstyle{definition}
\newtheorem{definition}[theorem]{Definition}

\newtheorem{xca}{Exercise}[section]
\newcommand{\ignore}[1]{}

\newcommand{\reals}{{\mathbb R}}

\newcommand{\integers}{{\mathbb Z}}

\newcommand{\complex}{{\mathbb C}}

\newcommand\EE{\mathbb{E}}

\newcommand\II{\mathbb{I}}

\DeclareMathAlphabet{\mathcal}{OMS}{cmsy}{m}{n}

\newcommand\cM{\mathcal{M}}

\newcommand\cP{\mathcal{P}}

\newcommand\cR{\mathcal{R}}

\DeclareMathOperator{\Prob}{Prob}
\def\from{\colon}

\def\eqdef{\overset{\text{def}}{=}}

\def\Ann{\qopname\relax o{Ann}}

\DeclareMathOperator{\Spec}{Spec}

\DeclareRobustCommand
  \rddots{\mathinner{\mkern1mu\raise\p@
    \vbox{\kern7\p@\hbox{.}}\mkern2mu
    \raise4\p@\hbox{.}\mkern2mu\raise7\p@\hbox{.}\mkern1mu}}

\newcommand\xhookrightarrow[2][]{\ext@arrow 0062{\hookrightarrowfill@}{#1}{#2}}
\def\hookrightarrowfill@{\arrowfill@\lhook\relbar\rightarrow}

\usepackage{relsize}
\usepackage{tikz}
\usetikzlibrary{matrix,arrows,decorations.pathmorphing}
\usepackage{tikz-cd}
\usetikzlibrary{cd}

\usepackage[pdftex,colorlinks,linkcolor=blue]{hyperref}

\begin{document}

\title[Relativized Alon Conjecture IV] % running head
{A Relativized Alon Second Eigenvalue
Conjecture for Regular Base Graphs IV: An Improved Sidestepping Theorem}

\author{Joel Friedman}
\address{Department of Computer Science,
        University of British Columbia, Vancouver, BC\ \ V6T 1Z4, CANADA}
\curraddr{}
\email{{\tt jf@cs.ubc.ca}}
\thanks{Research supported in part by an NSERC grant.}

\author{David Kohler}
\address{Department of Mathematics,
        University of British Columbia, Vancouver, BC\ \ V6T 1Z2, CANADA}
\curraddr{422 Richards St, Suite 170, Vancouver BC\ \  V6B 2Z4, CANADA}
\email{{David.kohler@a3.epfl.ch}}
% \email{{\tt david.emmanuel.kohler@gmail.com} or {David.kohler@a3.epfl.ch}}
\thanks{Research supported in part by an NSERC grant.}

% Day and time:
% \date{\today, at \currenttime  (get rid of time in final version)}
%
% Just day:
\date{\today}



\begin{abstract}

This is the fourth in a series of articles devoted to showing that a typical
covering map of large degree to a fixed, regular graph has its new adjacency
eigenvalues within the bound conjectured by Alon for random regular graphs.

In this paper we prove a {\em Sidestepping Theorem} that is more general and
easier to use than earlier theorems of this kind.  Such theorems concerns a
family probability spaces $\{\cM_n\}$ of $n\times n$ matrices, where $n$ varies
over some infinite set, $N$, of natural numbers.  Many trace methods use simple
``Markov bounds'' to bound the expected spectral radius of elements of $\cM_n$:
this consists of choosing one value, $k=k(n)$, for each $n\in N$, and proving
expected spectral radius bounds based on the expected value of the trace of the
$k=k(n)$-power of elements of $\cM_n$.  {\em Sidestepping} refers to bypassing
such simple Markov bounds, obtaining improved results using a number of values
of $k$ for each fixed $n\in N$.

In more detail, if the $M\in \cM_n$ expected value of $\Trace(M^k)$ has an
asymptotic expansion in powers of $1/n$, whose coefficients are ``well
behaved'' functions of $k$, then one can get improved 
bounds on the 
spectral radius of elements of $\cM_n$ that hold with high probability.
Such asymptotic expansions are shown
to exist in the third article in this series for the families of matrices that
interest us; in the fifth and 
sixth article in this series we will apply the Sidestepping
Theorem in this article to prove the main results in this series of articles.

This article is independent of all other articles in this series; it can be
viewed as a theorem purely in probability theory, concerning random matrices
or, equivalently, the $n$ random variables that are the eigenvalues of the
elements of $\cM_n$.

\end{abstract}

\maketitle
\setcounter{tocdepth}{3}
\tableofcontents

%%%%%%%%%%%%%%%%%%%%%%%%%%%%%%%%%%%%%%%%%%%%%%%%%%%%%%%%%% 
% temp:
\newcommand{\sePrelimProofs}{17}

\section{Introduction}

This is the fourth article in a series of six articles devoted to
developing trace methods to prove a relativization of the Alon
Second Eigenvalue Conjecture for covering maps of a fixed base
graph that is regular; we also get a sharper theorem when the
base graph is {\em Ramanujan}.
This article is independent of all the other articles and most of the
terminology they use; the results here are purely theorems in probability
theory.

The main goal of this article is to prove a strengthening of the 
{\em Sidestepping Lemma} of \cite{friedman_alon}, i.e., Lemma~11.1
there,
that allows us to infer---with very high probability---bounds
on the eigenvalues of random $n\times n$ matrices, provided that
the expected values of their $k$-power traces satisfy certain
``asymptotic expansions.''
Such asymptotic expansions are proven in Articles~I---III, the
first three articles in this series, for random
matrices that express the {\em new Hashimoto} matrix
(i.e., {\em new non-backtracking} matrix)
of random matrices of the random covering maps
to a fixed base graph.
Applying the main result in this article to those of
Articles~III (i.e., the third article in this series)
proves fundamental results about the 
location of the eigenvalues of such random matrices when the base
graph is regular.
This application forms the basis of Articles~V and VI of this series,
where the two main results of this series of articles are proven.

After proving the
{\em Sidestepping Lemma} in this article, akin to that of
\cite{friedman_alon}, 
we will give a consequence of this lemma
which we call the
{\em Sidestepping Theorem;} this theorem is simpler to state,
and suffices for our applications in 
Articles~V and VI.
Hence we state the Sidestepping Theorem as the main result in this 
paper, and it is this theorem that we will quote in Article~V.

The rest of this article is organized as follows.
In Section~\ref{se_main_results} we precisely state the Sidestepping
Theorem, as well as the two main lemmas that are used to prove it;
we call these lemmas the Exceptional Eigenvalue Bound and the
Sidestepping Lemma, and these lemmas have direct analogs in
\cite{friedman_alon}.
In Section~\ref{se_r0} we introduce a number of tools---most of which
are connected to the {\em shift operator}---and prove the
Exceptional Eigenvalue Bound.
In Section~\ref{se_r1} we prove the Sidestepping Lemma,
which uses the Exceptional Eigenvalue Bound and develops additional
properties of the shift operator.
In Section~\ref{se_side_theorem} we easily prove the Sidestepping Theorem
as a consequence of the Sidestepping Lemma.

For the rest of the section we motivate the Sidestepping Theorem in this
article and make some rough remarks on its proof.

\subsection{Motivation and Rough Idea Behind the Sidestepping Theorem}

The most basic trace methods use
a simple Markov-type bound as follows:
if $M$ is an $n\times n$ matrix with real or complex entries
and with real eigenvalues,
then if we use $\rho(M)$ to denote the spectral radius of 
$M$ (i.e., the largest absolute value among $M$'s
eigenvalues), 
then for any even integer $k>0$ we have
$$
\rho(M)^k \le \Trace(M^k);
$$
it follows that if $\cM$ is a probability space of such matrices, then
for any $\alpha>0$ 
\begin{equation}\label{eq_Markov_type}
\Prob_{M\in\cM}[ \rho(M) \ge \alpha ] \le \alpha^{-k} 
\EE_{M\in\cM}[\Trace(M^k)].
\end{equation} 
This is a simple Markov-type bound, used in
\cite{broder,friedman_random_graphs,friedman_relative,linial_puder,
puder} to bound the {\em new adjacency spectral radius} of random covers
of a base graph.  Unfortunately a direct Markov method like this
cannot yield the optimal eigenvalue bound conjectured to hold with
high probability by Alon (although \cite{puder} obtains results that
are quite close).
The problem is that certain {\em tangles} cause
$$
\EE_{M\in\cM}[\Trace(M^k)]
$$
to be too large for \eqref{eq_Markov_type} to be useful.  We use the
term ``sidestepping'' for results that avoid this direct Markov-type
approach.

In more detail, consider the situation
where $\cM_n$ is a probability space of such $n\times n$ matrices,
defined for an infinite number of positive integers, $n$, such that
for any some
integer $r>1$ one has an asymptotic expansion
\begin{equation}\label{eq_basic_idea}
\EE_{M\in\cM_n}[\Trace(M^k)] 
=f_0(k)+f_1(k)/n + \cdots + f_{r-1}(k)/n^{r-1}+O(1)f_r(k)/n^r
\end{equation}
valid for all $k\le n^{1/2}$.
To fix ideas, say that there are real $\Lambda_1>\Lambda_0>0$ such that
for $i\le r-1$ we have $f_i(k)=O(\Lambda_0^k)$
and that $f_r(k)=O(\Lambda_1^k)$.
Then choosing $k$ to balance the bounds on $f_0(k)$ and 
$f_r(k)/n^r$ (so $k$ will be proportional to $\log n$), 
the Markov-type bound \eqref{eq_Markov_type} gives 
a high probability bound for $\rho(M)$ of at most
$$
\Lambda_0^{(r-1)/r}\Lambda_1^{1/r},
$$
which is a function tending to $\Lambda_0$ as $r\to\infty$.

We remark that later in our main theorems, the upper bounds
$O(\Lambda_0^k)$ and $O(\Lambda_1^k)$ (that appear in this section)
will be replaced by
slightly weaker upper bounds; we write the bounds
$O(\Lambda_0^k)$ and $O(\Lambda_1^k)$ for simplicity, to illustrate
the main ideas.

The Sidestepping Theorem deals with the following situation: say that
for any $r$ we have an expansion
\eqref{eq_basic_idea}, but that for some $i<r$ we have that
$f_i(k)$ can be written as $O(\Lambda_0^k)$ plus a finite number
of terms of the form $C k^a\lambda^k$ for an integer $a$, $C\in\complex$, and
$\lambda$ is real with $\Lambda_0<|\lambda|\le \Lambda_1$; 
we call any sum of such terms {\em polyexponentials in $k$}.  
Then we can still infer a high probability bound
of $\Lambda_0+g(r)$ where $g(r)\to 0$ as $r\to \infty$, under appropriate
conditions that will be made precise in Section~\ref{se_main_results}.

Our proof uses the following ideas.
Let $S$ denote the ``shift operator in $k$,'' i.e., 
$S$ acts on function of
$k$ via $(S f)(k)\eqdef f(k+1)$.
Then the function $f(k)\eqdef C k^a\lambda^k$ as above
is annihilated by
the operator
$(S-\lambda)^{a+1}$.
It follows that there is some polynomial $P=P(S)$ that annihilates
all the polyexponential terms in the $f_i(k)$ for $i=0,\ldots,r-1$
in \eqref{eq_basic_idea}.
Any fixed such polynomial $P=P(S)$
takes a function in $k$ bounded by $O(\rho^k)$ to function bounded
by $O(k^d\rho^k)$ (where $d=\deg(P)$).
So now we apply $P(S)$ to both sides of
\eqref{eq_basic_idea} and try to infer information on the location
of the eigenvalues of $M\in\cM_n$.
Furthermore, it is easy to see that for any polynomial, $P$,
$$
P(S)(\lambda^k) = P(\lambda)\lambda^k
$$
(where $\lambda^k$ on the left-hand-side is shorthand for the
function $k\mapsto \lambda^k$); this allows us to control the 
effect of $P(S)$ on the expected value of $\Trace(M^k)$ provided
that $\lambda$ is real and $\Lambda_0<|\lambda|\le\Lambda_1$.

Of course, in the above one is using the fact that 
\eqref{eq_basic_idea} holds for many values of $k$ for a given $n$,
and hence one can apply polynomials in $S$ to
\eqref{eq_basic_idea}; furthermore, the left-hand-side
becomes the expected value of the sum of
$P(S)\lambda_i(M)^k$
over the eigenvalues, $\lambda_i(M)$, of $M$, rather than merely
$\lambda_i(M)^k$.
By contrast, the Markov-type bound only uses a single value of $k$.
[In Markov-type methods and our sidestepping methods, 
all the values of $k$
that we use are proportional to $\log n$.]

Our Sidestepping Theorem is stated somewhat more generally, because
the non-backtracking matrix or Hashimoto matrix of a $d$-regular graph
can also have complex eigenvalues, but only those of absolute value
$(d-1)^{1/2}$. 
In our applications, $\Lambda_1=d-1+\epsilon$ and 
$\Lambda_0=(d-1)^{1/2}+\epsilon$, where $d$ is fixed and
$\epsilon>0$ is small, and $f_i(k)$ begins to contain
polyexponential terms $Ck^a\lambda^k$ with $\lambda>\Lambda_0$ for
$i$ at least of order roughly $d^{1/2}$.

We remark that our Sidestepping Theorem concerns only the random variables
$\lambda_i(M)$, so that one could formulate this theorem without any
reference to matrices, only referring to the random variables
$\lambda_1(M),\ldots,\lambda_n(M)$ of $M\in\cM_n$.
However, we prefer to state the Sidestepping Theorem in (its equivalent
form) involving eigenvalues of matrices, because this is closer to our
intuition and is the context in which we apply the theorem.

\section{Main Results}
\label{se_main_results}

In this section we state the main results in the paper,
namely Lemmas~\ref{le_sidestep_one} and~\ref{le_sidestep_two},
and the result that we need for Article~V, namely
Theorem~\ref{th_sidestep}.
For convenience, we repeat some of the notation and definitions in
Article~I.

\subsection{Preliminary Notation}

In this subsection we introduce some notation that is completely standard,
except for possibly \eqref{eq_B_L} below.

We use $\reals,\complex,\integers,\naturals$
to denote, respectively, the
the real numbers, the complex numbers, the integers, and positive
integers or
natural numbers;
we use $\integers_{\ge 0}$ to denote the set of non-negative
integers.
If $n\in\naturals$ we use $[n]$ to denote $\{1,2,\ldots,n\}$.

If $\rho\in\reals$ is non-negative and $z\in\complex$, we use
$B_\rho(z)$ to denote the {\em closed}
ball of radius $\rho$ about $z$, i.e., 
the set $\{x\in\complex \ | \ |x-z|\le \rho\}$.
If $a,b\in\reals$, we use $[a,b]$ to denote the closed interval
$\{x\;|\; a\le x\le b\}\subset \reals$; we often use the
inclusion $\reals\subset\complex$ (as the complex numbers with vanishing
imaginary part), to view subsets of $\reals$ as subsets of $\complex$.

If $L\subset\complex$ and $\rho\ge 0$ is real, we use the 
(fairly standard) notation
\begin{equation}\label{eq_B_L}
B_\rho(L) = \bigcup_{\ell\in L} B_\rho(\ell)  \ .
\end{equation} 

All probability spaces are finite; hence a probability space
is a pair $\cP=(\Omega,P)$ where $\Omega$ is a finite set and
$P\from \Omega\to\reals_{>0}$ with $\sum_{\omega\in\Omega}P(\omega)=1$;
hence an {\em event} is any subset of $\Omega$.
We emphasize that $\omega\in\Omega$ implies that $P(\omega)>0$ with
strict inequality (i.e., no atoms have zero probability).
We use $\cP$ and $\Omega$ interchangeably when $P$ is
understood and confusion is unlikely.

A {\em complex-valued random variable} on $\cP$ or $\Omega$
is a function $f\from\Omega\to\complex$, and similarly for real-,
integer-, and natural-valued random variable; we denote its
$\cP$-expected value by
$$
\EE_{\omega\in\Omega}[f(\omega)]=\sum_{\omega\in\Omega}f(\omega)P(\omega).
$$
If $\Omega'\subset\Omega$ we denote the probability of $\Omega'$ by
$$
\Prob_{\cP}[\Omega']=\sum_{\omega\in\Omega'}P(\omega')
=
\EE_{\omega\in\Omega}[\II_{\Omega'}(\omega)],
$$
where $\II_{\Omega'}$ denotes the indicator function of $\Omega'$.
At times we write $\Prob_{\cP}[\Omega']$ and $\II_{\Omega'}$
where $\Omega'$ is
not a subset of $\Omega$, by which we mean
$\Prob_{\cP}[\Omega'\cap\Omega]$ and $\II_{\Omega'\cap\Omega}$.

A {\em matrix-valued random variable} (with entries in $\reals$ 
or $\complex$) is
similarly defined.

\subsection{$\EE{\rm in}$ and $\EE{\rm out}$}

If $\cM$ is a probability space of square matrices, and 
$R\subset \complex$, it will be very useful to define
\begin{equation}\label{eq_Ein_Eout}
\EE{\rm in}_\cM[R] \quad\mbox{and}\quad 
\EE{\rm out}_\cM[R] \ ,
\end{equation} 
respectively,
as the expected number of eigenvalues of $M\in\cM$ (counted with multiplicity)
that lie, respectively, in $R$ and not in $R$.  Of course,
these two expected values are non-negative; furthermore, if $\cM$ consists 
entirely of $n\times n$ matrices, then 
$$
\EE{\rm in}_\cM[R] + \EE{\rm out}_\cM[R] = n.
$$

\subsection{Polyexponentials and Approximate Polyexponentials}

\begin{definition}\label{de_polyexp}
We say that a function $Q\from\naturals\to\complex$ is a {\em polyexponential
function} 
if it is a linear combination of functions of the form
$$
f(k) = k(k-1)\ldots (k-m+1)\ell^{k-m}
$$
for some $m\in\integers_{\ge 0}$ and $\ell\in\complex$, or, equivalently,
if we have
\begin{equation}\label{eq_de_polyexp}
Q(k) = \sum_{\ell\in L}p_\ell(k) \ell^k
\end{equation} 
where $p_\ell=p_\ell(k)$ are polynomials and $L\subset\complex$ is a finite
set, and where the expression $p_\ell(k)\ell^k$ with $\ell=0$ is taken
to mean any function that vanishes for $k$ sufficiently large.
It is easy to see that
the representation \eqref{eq_de_polyexp} is unique if we insist that
$L$ is minimal, i.e.,
$\ell\in L$ appears only if it is needed, i.e.,
if $p_\ell(k)\ell^k$ is a nonzero function;
we refer to the minimal set $L$ as the {\em set of bases of $Q$},
and if $\ell\in L$ we refer to the function $p_\ell(k)\ell^k$ as
the {\em $\ell$-part} of $Q$.
\end{definition}

For example, if $n\in\naturals$, $i,j\in [n]$, and
$M$ is any $n\times n$ matrix over the complex numbers, then
$f(k)=(M^k)_{i,j}$ (the $i,j$-th entry of $M^k$) is a
polyexponential function of $k$.
% ; conversely, it is easy to see that
% any polyexponential function is a linear combination of such functions for
% a single $M$ (consisting of appropriate Jordan blocks) and pairs
% $i,j$ varying in $[n]$.
The fact that a Jordan block with eigenvalue $0$ is nilpotent justifies
our convention that for $\ell=0$, $p_\ell(k)\ell^k$ refers to any function
of finite support.

\begin{definition}
Let $\rho\in\reals$.
We say that a function $f\from\naturals\to\complex$ is
{\em of growth $\rho$} if for every $\epsilon>0$ we have that
$|f(k)|<(\rho+\epsilon)^k$ for $k$ sufficiently large.
By an {\em approximate polyexponential with error growth $\rho$} we mean
a function $Q\from\naturals\to\complex$
that can be written as $Q(k)=f(k)+g(k)$, where $g$ is of growth $\rho$ and
$f$ is a polyexponential;
given such a $Q$ and $\rho$, it is easy to see that
$f$ and $g$ are uniquely determined if
we insist that all bases of $f$ are greater than $\rho$ in absolute
value (if $|\ell|\le \rho$, we may move a term $p_\ell(k)\ell^k$
of $f$ to
$g$); we refer to this unique $f$ as the {\em polyexponential part} of
$Q$ (with respect to $\rho$), and to the
bases of $f$ as the {\em larger bases} of $Q$ (with respect
to $\rho$);
if $\ell$ is one of these larger bases, then the
{\em $\ell$-part of $Q$} is the $\ell$-part of $f$, i.e., the term
$p_\ell(k)\ell^k$ in the sum
$$
f(k)
=\sum_{\ell\in L}p_\ell(k) \ell^k .
$$
\end{definition}

\subsection{$(\Lambda_0,\Lambda_1)$-Bounded Matrix Models}

We now define the type of families of random matrices that are of interest
to us in this article.

\begin{definition}\label{de_matrix_model}
Let $\Lambda_0<\Lambda_1$ be positive real numbers.  
By a {\em $(\Lambda_0,\Lambda_1)$-bounded
matrix model} we mean a collection of finite probability spaces
$\{\cM_n\}_{n\in N}$ where $N\subset\naturals$ is an
infinite subset, and where the atoms of $\cM_n$ are
$n\times n$ real-valued matrices whose eigenvalues lie in the set
$$
B_{\Lambda_0}(0) \cup [-\Lambda_1,\Lambda_1]
% \bigl( B_{\Lambda_1}(0)\cap \reals \bigr)
$$
in $\complex$.
Let $r\ge 0$ be an integer and $K\from\naturals\to\naturals$
be a function such that $K(n)/\log n\to \infty$ as $n\to\infty$.
We say that this model has an
{\em order $r$ expansion} with {\em range $K(n)$} 
(with $\Lambda_0,\Lambda_1$ understood) if
% for some function $K=K(n)$ such that $K(n)/\log n\to \infty$ 
as $n\to\infty$ we have that
\begin{equation}\label{eq_matrix_model_exp}
\EE_{M\in\cM_n}[\Trace(M^k)] =
c_0(k) + c_1(k)/n +\cdots+c_{r-1}(k)/n^{r-1}+ O(c_r(k))/n^r
\end{equation}
for all $n\in N$ and integers $k$ with $1\le k\le K(n)$,
where (1) $c_r=c_r(k)$ is of growth $\Lambda_1$,
(2) the constant in the $O(c_r(k))$ is
independent of $k$ and $n$, and
(3) for $0\le i<r$, $c_i=c_i(k)$ is an approximate polyexponential with
$\Lambda_0$ error term and whose larger bases 
(i.e., larger than $\Lambda_0$ in absolute value)
lie in $[-\Lambda_1,\Lambda_1]$;
at times we speak of an {\em order $r$ expansion} without 
explicitly specifying $K$.
When the model has such an expansion,
then we use the notation $L_r$ to refer to the union of all larger
bases of $c_i(k)$ (with respect to $\Lambda_0$) over all $i$ between
$0$ and $r-1$, and call $L_r$ the {\em larger bases (of the order $r$
expansion)}.
\end{definition}

[We remark that since $\Lambda_0$ is positive, it is easy to see
that the expansion
\eqref{eq_matrix_model_exp}
also holds for $k=0$ (by possibly redefining $c_r(0)$ so that it is
positive); we prefer to insist that $k\ge 1$ so that in Articles~I--III
we don't need to worry about certain technicalities when we define
{\em algebraic models} there (such as
the empty graph and non-backtracking
walks of length $0$).]

The main theorems in this paper concern a single, fixed matrix model
(with fixed $\Lambda_0,\Lambda_1$ as above), and we will reserve
the symbols $\cM_n$ and $L_r$ (and $\Lambda_0,\Lambda_1$) to have
the above meaning.
We also note that \eqref{eq_matrix_model_exp} implies that
for fixed $k\in\naturals$,
\begin{equation}\label{eq_c_i_limit_formula}
c_i(k) = \lim_{n\in N,\ n\to\infty}
\Bigl( \EE_{M\in\cM_n}[\Trace(M^k)] -
\bigl( c_0(k) + \cdots+c_{i-1}(k)/n^{i-1} \bigr) \Bigr) n^i
\end{equation} 
for all $i\le r-1$; we conclude that the $c_i(k)$ are uniquely determined,
and that $c_i(k)$ is independent of $r$ for any $r>i$ for which
\eqref{eq_matrix_model_exp} holds.
We also see that if \eqref{eq_matrix_model_exp} holds for some value of
$r$, then it also holds for smaller values of $r$.
It follows that if \eqref{eq_matrix_model_exp} holds for some $r$, then
$L_i$ is defined for each $i<r$ (as the set of larger bases of the
functions $c_0(k),\ldots,c_{i-1}(k)$).
Furthermore $L_i$ is empty iff
$c_0(k),\ldots,c_{i-1}(k)$ are all functions of growth $\Lambda_0$.

We remark that in our applications (see Article~V), 
$N$, the set of $n$ for which
$\cM_n$ is defined, will generally be a proper subset of $\naturals$;
indeed, $n$ will be multiples of the number of directed edges of a
fixed {\em base graph}, $B$, since $\cM_n$ are related to the
Hashimoto matrix,
$H_G$, of random covers, $G$, of $B$.

For brevity we often write $n\to \infty$ or $\lim_{n\to\infty}$
understanding that $n$ is confined to the set $N$ for which $\cM_n$ is
defined.
Similarly, when we write an equation involving $n$, especially
one involving $\cM_n$, we understand $n$ to be confined to $N$.

\subsection{The Sidestepping Theorem}

For Article~V we need only the following theorem.
This theorem will be proven in two main steps,
represented in Lemmas~\ref{le_sidestep_one} 
and~\ref{le_sidestep_two} below; the theorem follows almost
immediately from Lemma~\ref{le_sidestep_two}.

\begin{theorem}\label{th_sidestep}
Let $\{\cM_n\}_{n\in N}$ be a $(\Lambda_0,\Lambda_1)$-bounded matrix model,
for some real $\Lambda_0<\Lambda_1$, that
for all $r\in\naturals$ has an order $r$ expansion;
let $p_i(k)$ denote the polyexponential part of $c_i(k)$
(with respect to $\Lambda_0$) in \eqref{eq_matrix_model_exp}
(which is independent of $r\ge i+1$ by
\eqref{eq_c_i_limit_formula}).
If $p_i(k)=0$ for all $i\in\integers_{\ge 0}$, then
for all $\epsilon>0$ and $j\in\integers_{\ge 0}$
\begin{equation}\label{eq_largest_j}
\EE{\rm out}_{\cM_n}
\bigl[ B_{\Lambda_0+\epsilon}(0) \bigr]
= O(n^{-j}) .
\end{equation}
Otherwise let $j$ be the smallest integer for which
$p_j(k)\ne 0$.
Then for all $\epsilon>0$, and for all $\theta>0$ sufficiently small we have
% \eqref{eq_largest_j} holds;
\begin{equation}\label{eq_mostly_near_Lambda_0_or_Ls}
\EE{\rm out}_{\cM_n}
\bigl[ B_{\Lambda_0+\epsilon}(0)\cup B_{n^{-\theta}}(L_{j+1}) \bigr]
= o(n^{-j});
\end{equation} 
moreover, if $L=L_{j+1}$ is the (necessarily nonempty) set of bases of $p_j$,
then for each $\ell\in L$ there is a real $C_\ell>0$ such that
\begin{equation}\label{eq_thm_p_j_pure_exp}
p_j(k)=\sum_{\ell\in L} \ell^k C_\ell ,
\end{equation} 
and
for all $\ell\in L$
for sufficiently small $\theta>0$,
\begin{equation}\label{eq_thm_C_ell_as_limit}
\EE{\rm in}_{\cM_n}\bigl[ B_{n^{-\theta}}(\ell) \bigr]  
= n^{-j} C_\ell + o(n^{-j}) .
\end{equation} 
\end{theorem}
In the above theorem, $L=L_{j+1}$ as in Definition~\ref{de_matrix_model},
and this definition implies that
$$
L\ \subset
\ [-\Lambda_1,\Lambda_1] \setminus
[-\Lambda_0,\Lambda_0].
$$
We remark that in our applications $p_i(k)\ne 0$ for some $i$;
the theorem then states that the smallest $j$ 
for which $p_j(k)\ne 0$ corresponds to the power of $1/n$ such that
eigenvalues begin to appear with probability order $n^{-j}$
near the bases of $p_j(k)$.
% Furthermore, given the assumptions of the above theorem,
% Lemma~\ref{le_sidestep_one} below implies
% that, roughly speaking, with probability larger than
% any power of $1/n$, the eigenvalues of 
% $\cM_n$ lie in $B_{\Lambda_0+\epsilon}(0)$ or very near the bases
% of all the $p_i(k)$.

\subsection{The Exceptional Eigenvalue Bound}

The first step in proving the Sidestepping Theorem is what we
call the {\em Exceptional Eigenvalue Bound}; this was called the
{\em First Exception Bound} in Section~2.6 of \cite{friedman_kohler}
(see Subsection~2.6.3 there).
It is interesting in its own right.

\begin{lemma}\label{le_sidestep_one}
Let $\epsilon,\alpha,\Lambda_0,\Lambda_1>0$ be positive real numbers
with $\Lambda_0<\Lambda_1$.
Then there is
a positive integer $r_0=r_0(\Lambda_0,\Lambda_1,\epsilon,\alpha)$ 
and a real
$\theta_0=\theta_0(\Lambda_0,\Lambda_1,\epsilon,\alpha)>0$
with the following property:
let $\{\cM_n\}_{n\in N}$ be a $(\Lambda_0,\Lambda_1)$-bounded matrix model
that has an order $r$ expansion \eqref{eq_matrix_model_exp} 
for some $r\ge r_0$, and let $L_r$ be the set of larger bases of 
$c_0(k),\ldots,c_{r-1}(k)$.
Then for any $\theta\le\theta_0$,
\begin{equation}\label{eq_sidestep_r_one}
\EE{\rm out}_{\cM_n}
\bigl[ B_{\Lambda_0+\epsilon}(0) \cup B_{n^{-\theta}}(L_r) \bigr] 
\le n^{-\alpha} 
\end{equation} 
for $n$ sufficiently large.
\end{lemma}
In other words, an eigenvalue of $M\in\cM_n$ that lies outside of
$B_{\Lambda_0+\epsilon}(0) \cup B_{n^{-\theta}}(L_r)$ is 
``exceptional'' in the sense that the expected number
of such eigenvalues decays larger than any fixed power of $n$
(for $\epsilon,\theta>0$ sufficiently small),
for models with expansions of sufficiently large
order $r$.

Of course, if \eqref{eq_sidestep_r_one} holds for $\theta,r$, then
it holds for any larger $r$ and smaller $\theta>0$ since 
$B_{n^{-\theta}}(L_r)$ is non-decreasing as $r$ increases and
$\theta$ decreases.

We remark that to prove the above lemma we will take 
$r_0=r_0(\Lambda_0,\Lambda_1,\epsilon,\alpha)$ to be any integer such
that
\begin{equation}\label{eq_give_r_zero}
r_0 >
\alpha+ (\alpha+1)
\frac{\log(\Lambda_1) - \log(\Lambda_0+\epsilon)}{
\log(\Lambda_0+\epsilon) - \log(\Lambda_0)} 
\end{equation}
(see \eqref{eq_first_ex_kappa} and \eqref{eq_first_ex_r} in the
next Section).
Hence for $\Lambda_0,\Lambda_1,\alpha$ fixed and $\epsilon$ small, which
is the case in our applications, we
have that $r_0$ is proportional to $1/\epsilon$, in view of the denominator
in the fraction in \eqref{eq_give_r_zero}.
Hence for small $\epsilon$ we need an asymptotic expansion of order
proportional to $1/\epsilon$ to be able to apply
the above lemma.

\subsection{The Sidestepping Lemma}

We now state the main lemma of this article, which is a version of the
Sidestepping Lemma of \cite{friedman_alon} that is more general and
easier to apply in our situation.
It easily implies Theorem~\ref{th_sidestep}, which is easier to state
and sufficient for us in Articles~V and~VI.

\begin{lemma}\label{le_sidestep_two}
Let $\Lambda_0<\Lambda_1$ and $\epsilon$ be positive real numbers
with $\Lambda_0+\epsilon\le \Lambda_1$,
and let $j\ge 0$ be an integer.
Then there are $r_1=r_1(\Lambda_0,\Lambda_1,j,\epsilon)$ 
and $\theta_1=\theta_1(\Lambda_0,\Lambda_1,j,\epsilon)$
with the following properties:
let $\{\cM_n\}_{n\in N}$ be a $(\Lambda_0,\Lambda_1)$-bounded matrix model
that has an order $r$ expansion \eqref{eq_matrix_model_exp} for 
some $r\ge r_1(\Lambda_0,\Lambda_1,j,\epsilon)$.
Recall that 
for each $i<r$, $L_i$ denotes the larger bases of
$c_0(k),\ldots,c_{i-1}(k)$ (with respect to $\Lambda_0$).
If $L_j=\emptyset$,
then
\begin{equation}\label{eq_most_useful}
\EE{\rm out}_{\cM_n}[B_{\Lambda_0+\epsilon}(0)] = O(n^{-j}).
\end{equation}
In more detail, for any positive $\theta\le\theta_1$ we have
\begin{equation}\label{eq_exception_for_side_stepping}
\EE{\rm out}_{\cM_n}
\bigl[ B_{\Lambda_0+\epsilon}(0)\cup B_{n^{-\theta}}(L_{j+1}) \bigr]
= o(n^{-j}).
\end{equation} 
Furthermore, 
for each $\ell\in L_{j+1}$ with $|\ell|\ge \Lambda_0+2\epsilon/3$
we have
\begin{equation}\label{eq_C_ell_alternate_formula}
\EE{\rm in}_{\cM_n}\bigl[ B_{n^{-\theta}}(\ell) \bigr] 
= C_\ell n^{-j} + o(n^{-j}),
\end{equation}
for some real $C_\ell\ge 0$.
Furthermore, if $q_{\ell,j}=q_{\ell,j}(k)$ are the polynomials in $k$
such that
$$
p_j(k) = \sum_{\ell\in L_{j+1}} q_{\ell,j}(k)\ell^k
$$
is 
the polyexponential part of $c_j(k)$ (with respect to 
$\Lambda_0$), then if $|\ell|\ge \Lambda_0 + 2\epsilon/3$ we have
\begin{equation}\label{eq_q_ell_j_formula}
q_{\ell,j}(k) = C_\ell
\end{equation} 
is a constant function in $k$.
\end{lemma}
It other words, if $c_0(k),\ldots,c_{j-1}(k)$ are of growth
$\Lambda_0$, then we know the 
``expected location'' of the eigenvalues of $M\in\cM_n$ 
``to order $n^{-j}$'' in terms of the polyexponential part, $p_j(k)$, of 
$c_j(k)$ with respect to $\Lambda_0$ (this requires
that $\epsilon$ above
is taken small enough so that all the larger bases 
(with respect to $\Lambda_0$) of $c_j$ are
larger than $\Lambda_0+2\epsilon/3$);
furthermore, this $p_j(k)$ is a linear combination of 
purely exponential
functions of $k$, and the coefficients in this linear combination reflects
the excepted number (times $n^{-j}$)
of eigenvalues near each base of $p_j(k)$.
Of course, it is possible that $p_j(k)$ is identically zero, i.e.,
also $L_{j+1}$ is empty;
in such a situation the hypotheses of the
Sidestepping Lemma also hold when $j$ is replaced
with $j+1$, and we get more precise results provided that
$\cM_n$ has expansions of sufficiently large order.

We remark that in the above lemma,
if $\theta\le 0$, then the sets $B_{n^{-\theta}}(\ell)$ 
may contain elements of $L_{j+1}$ other than $\ell$; for this reason
we generally need $\theta>0$ in order for \eqref{eq_C_ell_alternate_formula}
to hold.

We remark that in our applications, our models $\cM_n$ can be assumed
to have expansions of arbitrarily high order, 
and so the particular value of $r_1(\Lambda_0,\Lambda_1,j,\epsilon)$
above is unimportant.

\section{The Shift Operator and The Proof of Lemma~\ref{le_sidestep_one}}
\label{se_r0}

In this section we develop some simple properties of the shift operator.
Then we separate $\Trace(M^k)$ into the contribution given by its
real and non-real eigenvalues, and develop some fundamental facts
about the $M\in\cM_n$-expected contribution due to the real
eigenvalues.
This allows us to prove Lemma~\ref{le_sidestep_one}, in a way that
we will easily adapt in Section~\ref{se_r1} to prove 
Lemma~\ref{le_sidestep_two}.
The main idea in both proofs, just as in
\cite{friedman_alon,friedman_kohler},
is to apply polynomials of the {\em shift in $k$} operator
to both sides of \eqref{eq_matrix_model_exp}.
These polynomials of the shift operator will be used to annihilate
the polyexponential parts of the coefficients of the asymptotic
expansion; one simply needs to keep track of 
how these shift-operator polynomials will affect the expected value
of $\Trace(M^k)$ and the other terms in \eqref{eq_matrix_model_exp}.

\subsection{The Shift Operator}

\begin{definition}
By the {\em shift operator in $k$}, denoted $S=S_k$, we mean
the operator
on functions, \(f=f(k)\), defined on non-negative integers, \( k \), taking
real or complex values, that takes
\(f\) to $Sf$ given by
\[ 
(Sf)(k) = f(k+1). 
\] 
For any integer \( i \geq 0 \) we let
\( S^i \) denote the \(i\)-fold application of \(S\), so that 
\[ (S^i
f)(k) = f(k+i). \] 
If $Q(z) = q_0 + q_1 z + \cdots + q_t z^t$ is any polynomial,
with real or complex coefficients, we define
$$
Q(S) = q_0 S^0 + q_1 S^1 + \cdots +
q_t S^t \ .
$$
If $f=f(k,n)$ depends on two variables, $k$ and $n$, we apply
$S$ and $S^i$ to $f(k,n)$ understanding that the shift is applied to
the first variable, with the second variable, $n$, fixed.
\end{definition}
Throughout this paper, $S=S_k$, will always be used to shift in $k$;
hence we never write the subscript $k$.

We remark that we often write formulas such as
\begin{equation}\label{eq_abusive}
S(\mu^k) = \mu\, \mu^k
\end{equation} 
with $k\in\naturals$ being a specific value; it would be more correct to write
\begin{equation}\label{eq_tedious}
\bigl( S(\mu^k) \bigr) \bigm|_{k=k_0} = \mu\,\mu^{k_0},
\end{equation} 
for a specific $k_0\in\naturals$, because the $\mu^k$ in $S(\mu^k)$
refers to a function of $k$.  However, 
we generally use \eqref{eq_abusive} for a particular value of $k$
rather than the more tedious \eqref{eq_tedious}, unless confusion is
likely.

There are a number of easy and useful results on 
the polynomials of shift operator, which we leave to the reader:
\begin{enumerate}
\item if $Q(z)=Q_1(z)Q_2(z)$ are polynomials,
then $Q(S)=Q_1(S)Q_2(S)=Q_2(S)Q_1(S)$;
\item if $Q(z) = q_0 + q_1 z + \cdots + q_t z^t$, then
\begin{equation}\label{eq_apply_shift_exponential}
\bigl( Q(S) \bigr)(\mu^k) = Q(\mu) \mu^k 
\end{equation} 
(where $\mu^k$ refers to this function of $k$);
\item if $p=p(k)$ is any polynomial, then
\begin{equation}\label{eq_annihilate_p_mu}
(S-\mu)^D \bigl( p(k)\mu^k \bigr) = 0 \quad\mbox{if}\quad
D> \deg(p).
\end{equation} 
\end{enumerate}
We will be interested in one particular type of polynomial.

\begin{definition}
\label{de_D_L_annihilator}
Let $D\in\integers$ and $L\subset\complex$ be a finite set.  We define the
{\em annihilator of $L$ of degree $D$} to be the polynomial
$$
\Ann_{D,L}(z) \eqdef  \prod_{\ell\in L} (z-\ell)^D.
$$
\end{definition}

We easily verify the following claims:
\begin{enumerate}
\item If $p(k)$ is any polyexponential function
whose bases lie in a finite set $L\subset\complex$, then 
$\Ann_{D,L}(S)p=0$ for $D\in\naturals$ sufficiently large
(in view of \eqref{eq_annihilate_p_mu}).
\item If $f(k)$ is a function of growth $\rho$, then $Sf$ is also of
growth $\rho$, since for any fixed $\epsilon>0$, for $k$ sufficiently
large we have
\begin{equation}\label{eq_Sf_also_growth_rho}
|f(k+1)|<(\rho+\epsilon/2)^{k+1} \le (\rho+\epsilon/2)^k(\rho+\epsilon/2)
\le (\rho+\epsilon)^k
\end{equation} 
for $k$ sufficiently large;
hence for any fixed $i\in\naturals$, 
$S^i f$ is also of growth $\rho$; hence 
for any fixed polynomial, $Q$,
$Q(S)f$ is also of growth $\rho$.
\item If $L\subset\complex$ is a finite set that is
closed under conjugation, and $\mu\in\reals$,
then $\Ann_{1,L}(\mu)$ is real-valued, since
$$
\overline{\Ann_{1,L}(\mu)} = \prod_{\ell\in L} \overline{\mu-\ell}
= \prod_{\ell\in L} \bigl(\mu-\overline\ell\bigr) 
= \prod_{\ell\in L} \bigl(\mu-\ell\bigr) = \Ann_{1,L}(\mu);
$$
if, in addition, $D\in\naturals$ is even then we have
$$
\Ann_{D,L}(\mu) = \bigl( \Ann_{1,L}(\mu) \bigr)^D \ge 0;
$$
in particular, under these assumptions
\begin{equation}\label{eq_positivity}
\Ann_{D,L}(\mu^k) = \Ann_{D,L}(\mu) \,\mu^k \ge 0 
\end{equation}
when evaluated at any even integer $k$;
this positivity will be crucial to our proof of 
Lemma~\ref{le_sidestep_one}.
\end{enumerate}

\subsection{Real Eigenvalue Bounds}

For each $M\in\cM_n$, let
\begin{equation}\label{eq_real_non_real}
{\rm RealTrace}(M,k), \quad
{\rm NonRealTrace}(M,k)
\end{equation}
respectively denote the sum of the $k$-powers of the real and, respectively,
non-real eigenvalues of $M$.
Our proof of \eqref{eq_sidestep_r_one} is based on the following
two conceptual steps which we state as lemmas;
both lemmas will be proven in later subsections.

\begin{lemma}\label{le_real_eigs_bound}
Let $\{\cM_n\}_{n\in N}$ be a $(\Lambda_0,\Lambda_1)$-bounded matrix model
that has an order $r$ expansion with range $K(n)$
(\eqref{eq_matrix_model_exp}, Definition~\ref{de_matrix_model}).
Let $L\subset\complex$
be any finite set including all the
bases of $c_0(k),\ldots,c_{r-1}(k)$ in an expansion
\eqref{eq_matrix_model_exp}.
For any $D$ such that $\Ann_{D,L}$ annihilates the
polyexponential parts of $c_i(k)$ with respect to $\Lambda_0$
(therefore any sufficiently large $D$), we have
that for all $n\in N$ and integers $k$ with $0\le k\le K(n)-D(\#L)$.
\begin{equation}\label{eq_le_real_trace_bound}
\biggl|
\Ann_{D,L}(S) \Bigl( \EE_{M\in\cM_n}\bigl[ {\rm RealTrace}(M,k) \bigr]  \Bigr)
\biggr|
\le f_0(k)n+f_1(k)n^{-r},
\end{equation} 
where $f_0,f_1$ (depending on $D,L,r$)
are functions of respective growths $\Lambda_0,\Lambda_1$.
\end{lemma}
The proof of this lemma will be given in 
Subsection~\ref{su_proof_of_first_conceptual_lemma}.

We remark that although
\cite{broder,friedman_random_graphs,friedman_relative,linial_puder} 
work with adjacency
matrices, one could get similar estimates (for regular graphs)
by working with Hashimoto
matrices; if we work with Hashimoto matrices,
then these older works would essentially use
the above lemma in the (simple) special case where $D=0$, where
the $c_0,\ldots,c_{r-1}$ have vanishing
polyexponential part.  Hence one can view the above lemma as a generalization
of a lemma used by
older trace methods, where the older methods
use the special
case $D=0$.

It is easy to see roughly why the above lemma is true: we write
$$
{\rm RealTrace}(M,k) + {\rm NonRealTrace}(M,k) = \Trace(M^k)
$$
whose expected value
has a power series expansion, namely the right-hand-side of
\eqref{eq_matrix_model_exp}.  
The non-real trace term is bounded by $n$ (the maximum possible
number of eigenvalues non-real eigenvalues)
times $\Lambda_0^n$.
We apply $\Ann_{D,L}(S)$ where $D$ is 
large enough to annihilate the polyexponential parts of the $c_i$;
this leaves (1) $\Ann_{D,L}(S)$ applied to the $c_i(k)/n^i$, which
is $1/n^i$ times a function of growth $\Lambda_0$, plus 
(2) $\Ann_{D,L}(S)$ applied to the expected non-real trace, which is
bounded by $n$ times another function of growth $\Lambda_0$, plus
(3) $\Ann_{D,L}(S)$ applied to $c_r(k)/n^r$ which is $1/n^r$ times a 
function of growth $\Lambda_1$.
To prove the above lemma rigorously one has to verify that
$\Ann_{D,L}(S)$ acts in this way on the various terms.

\begin{lemma}\label{le_real_eigs_hp_bound}
Let $\{\cM_n\}_{n\in N}$ be a $(\Lambda_0,\Lambda_1)$-bounded matrix model.
% that has an order $r$ expansion with range $K(n)$
% (\eqref{eq_matrix_model_exp}, Definition~\ref{de_matrix_model}).
Let $L\subset\complex$ be any finite set closed under conjugation,
and let both $D$ and $k$ be positive even integers.
Then for any $\theta,\epsilon>0$,
for all $n\in\naturals$ and $k\in\integers$,
% for all $n\in N$ and integers $k$ with $0\le k\le K(n)-D(\#L)$,
% (with $K$ as in Definition~\ref{de_matrix_model}) we have
\begin{equation}\label{eq_E_n_theta_ineq}
n^{-\theta D (\# L) } (\Lambda_0+\epsilon)^k 
\EE{\rm out}_{\cM_n}\bigl[ B_{\Lambda_0+\epsilon}(0)\cup B_{n^{-\theta}}(L) \bigr]
\le
\Ann_{D,L}(S) \Bigl( \EE_{M\in\cM_n}\bigl[ {\rm RealTrace}(M,k) \bigr]  \Bigr)
\end{equation} 
(and both sides of the inequality are purely real).
% where $L_r$ is the set of larger bases of $c_0(k),\ldots,c_{r-1}(k)$
% in \eqref{eq_matrix_model_exp}.
\end{lemma}
The proof of this lemma will be given in 
Subsection~\ref{su_proof_of_second_conceptual_lemma}.
Note that this lemma does not require that $\{\cM_n\}$ have any
asymptotic expansion; its proof uses only the location of the
eigenvalues of the matrices in $\cM_n$.

% \begin{proof}
% For any $\lambda\in\reals$, 
% \eqref{eq_positivity} implies that $\Ann_{D,L}(\lambda^k)$ is 
% non-negative for $k,D$ even; hence the right-hand-side is real and
% non-negative; furthermore it is bounded below by the expected value of
% $\Ann_{D,L}(\lambda^k)$ taken over any eigenvalues, $\lambda$, of
% $\cM_n$ restricted to any subset of the $\reals$.
% If $\lambda$ is an eigenvalue of any element of $\cM_n$ that lies outside 
% of
% $$
% B_{\Lambda_0+\epsilon}(0)\cup B_{n^{-\theta}}(L),
% $$
% then $\lambda$ is necessarily real and satisfies 
% $$
% \bigl| \Ann_{D,L}(S)\lambda^k \bigr|
% =
% |\lambda|^k \prod_{\ell\in L} |\lambda-\ell|^D
% \ge (\Lambda_0+\epsilon)^k \prod_{\ell\in L} n^{-\theta D} ,
% $$
% which proves \eqref{eq_E_n_theta_ineq}.  Each factor in the
% left-hand-side of \eqref{eq_E_n_theta_ineq} is clearly 
% non-negative.
% \end{proof}

Once we have proven the above two lemmas, we will combine them to conclude
that
$$
n^{-\theta D (\# L) } (\Lambda_0+\epsilon)^k
\EE{\rm out}_{\cM_n}\bigl[ B_{\Lambda_0+\epsilon}(0)\cup B_{n^{-\theta}}(L) \bigr]
\le f_0(k)n+f_1(k)n^{-r},
$$
with the notation and conditions of the above lemmas;
then, with $D,L$ fixed, we will judiciously choose a fixed $r$ and
variable $k=k(n)$ to prove the Exceptional Eigenvalue Bound
(Lemma~\ref{le_sidestep_one}).

\subsection{Intermediate Lemmas} 

In this subsection we state some lemmas that will be used both to
prove Lemma~\ref{le_real_eigs_bound} and a variant thereof needed to
complete the proof of the Sidestepping Lemma.  They are conceptually
helpful and very easy to prove.

% \begin{lemma}
% If $c=c(k)\from\naturals\to \complex$ is of growth $\Lambda\in\reals$,
% then for any fixed $i\in\naturals$ 
% the function $\tilde c$
% defined by
% $\tilde c(k)=c(k+i)$ is also of growth $\Lambda$.
% \end{lemma}
% \begin{proof}
% For any $\epsilon>0$ we have $|c(k)|\le (\Lambda+\epsilon/2)^k$
% for $k$ sufficiently large, and hence
% $$
% |c(k+i)| \le (\Lambda+\epsilon/2)^k (\Lambda+\epsilon/2)^i 
% $$
% for $k$ sufficiently large.  Furthermore, for $k$ sufficiently large
% we have
% $$
% (\Lambda+\epsilon/2)^i
% \le \left( \frac{\Lambda+\epsilon}{\Lambda+\epsilon/2} \right)^k
% $$
% since the right-hand-side above tends to infinity as $k\to\infty$.
% Combining the above two displayed equations shows that
% $|c(k+i)|\le(\Lambda+\epsilon)^k$ for $k$ sufficiently large.
% \end{proof}

\begin{lemma}\label{le_c_i}
Let $c=c(k)\from\naturals\to\complex$ 
be an approximate polyexponential function with error growth
$\Lambda_0\in\reals$, whose larger
bases (with respect to $\Lambda_0$)
lie in a finite set $L\subset\complex$.  Then for sufficiently
large $D$ we have $\Ann_{D,L}(S)c(k)$ is of growth $\Lambda_0$.
\end{lemma}
\begin{proof}
We have $c(k)=f(k)+g(k)$ where $f$ is a polyexponential function
whose bases lie in $L$, and $g$ is of growth $\Lambda_0$.
% Since $\Ann_{D,L}(S)$ contains a factor
% of $(S-\ell)^D$ for each $\ell\in L$, 
If $D$ is sufficiently large then $\Ann_{D,L}(f)=0$ (by the first claim
after Definition~\ref{de_D_L_annihilator}).
Therefore for such $D$ we have
$$
\Ann_{D,L}(S)c(k) = \Ann_{D,L}(S)g(k),
$$
and the right-hand-side is a function of growth $\Lambda_0$, according
to the discussion around \eqref{eq_Sf_also_growth_rho}.
\end{proof}

Notice that if $c(k)=1$ for all $k$, then $c$ is of growth $1$; we easily check
that
$(S+1)^D c$ is the function $2^D c(k)=2^D$, which is not uniformly bounded
in $D$.
So it is important to understand
that in all the lemmas in this section, $D,L$ are arbitrary
but fixed, and the growth bounds are not uniform in $D$; for similar
reasons the
bounds are not uniform in $L$: taking $c$ as above and
$L=\{-1,-2,\ldots,-m\}$, we have $\Ann_{1,L}c = (m+1)!$.

\begin{lemma}\label{le_ann_c_r}
Let $r\ge 0$ be an integer, $K\from\naturals\to\naturals$ a function,
$N\subset\integers$ a subset, 
and $c=c(k)\from\naturals\to\reals_{\ge 0}$ 
a function of growth 
$\Lambda_1\in\reals$.
Let $f=f(k,n)$ be any function $\naturals\times N\to\complex$
such that $|f(k,n)|\le c(k)n^{-r}$ provided that $k,n\in\naturals$
satisfy $n\in N$ and $1\le k \le K(n)$.
Then for any finite set $L\subset\complex$ and any integer $D\ge 0$,
there is a function $g=g(k)$ of growth $\Lambda_1$ such that
$$
\bigl| \bigl( \Ann_{D,L}(S)f\bigr) (k,n) \bigr| \le g(k)n^{-r}
$$
provided that $n\in N$ and $1\le k \le K(n)-D(\#L)$.
\end{lemma}
\begin{proof}
We have
$$
\Ann_{D,L}(z) = a_0 + a_1 z + \cdots + a_t z^t
$$
for some $a_i\in \complex$ and $t=D(\#L)$.  For a fixed $n\in N$, the triangle
inequality shows that
$$
\bigl| \bigl( \Ann_{D,L}(S)f\bigr) (k,n) \bigr| 
\le \bigl( |a_0|\, |f(k,n)|+ |a_1|\, |f(k+1,n)| +\cdots
+|a_t|\,|f(k+t,n)| \bigr); 
$$
provided that $1\ge k$ and $k+t\le K(n)$, the right-hand-side above is 
bounded above by
$$
\le \bigl( |a_0|\, |c(k)|+ |a_1| \, |c(k)| 
+\cdots+|a_t|\,|c(k+t)| \bigr)/n^{-r} .
$$
But for each fixed $i\in\naturals$, 
$c(k+i)$ is also a function of growth $\Lambda_1$
(by the discussion around \eqref{eq_Sf_also_growth_rho}).
\end{proof}

\begin{lemma}\label{le_nonreal_eigs_bound}
Let $\{\cM_n\}_{n\in N}$ be a $(\Lambda_0,\Lambda_1)$-bounded matrix model
with range $K(n)$.
Then for any finite set $L\subset \complex$ and integer $D\ge 0$ we have that
for all $n\in\integers$ and $k$ with $1\le k\le K(n)-D(\#L)$,
\begin{equation}\label{eq_le_non_real_trace_bound}
\biggl|
\Ann_{D,L}(S) 
\Bigl( \EE_{M\in\cM_n}\bigl[ {\rm NonRealTrace}(M,k) \bigr]   \Bigr)
\biggr|
\le f_0(k)n,
\end{equation} 
where $f_0$ is a function of growth $\Lambda_0$.
\end{lemma}
\begin{proof}
For each $M\in\cM_n$ we have
$$
{\rm NonRealTrace}(M,k) = \sum_{\mu\in {\rm NonReal}(M)} \mu^k
$$
where ${\rm NonReal}(M)$ are the eigenvalues of $M$ that are not real.
Hence
$$
\Ann_{D,L}(S)
{\rm NonRealTrace}(M,k) = 
\sum_{\mu\in {\rm NonReal}(M)} \Ann_{D,L}(\mu)\mu^k,
$$
and hence
\begin{equation}\label{eq_nonreal_each_M}
\Bigl| \Ann_{D,L}(S) {\rm NonRealTrace}(M,k) \Bigr| \le
M n \Lambda_0^k,
\end{equation} 
where
$$
M = \max_{|z|\le\Lambda_0} \bigl| \Ann_{D,L}(z) \bigr|.
$$
But the triangle inequality implies that
for any random variable $g\from\cM_n\to\complex$ we have
$$
\bigl| \EE_{M\in \cM_n}g(M) \bigr| \le
\EE_{M\in \cM_n} |g(M)|;
$$
taking $g=\Ann_{D,L}(S) {\rm NonRealTrace}(M,k)$ and
exchanging the order of $\Ann_{D,L}(S)$ and $\EE_{M\in\cM_n}$
(which clearly commute),
\eqref{eq_nonreal_each_M} implies that
$$
\biggl|
\Ann_{D,L}(S) 
\Bigl( \EE_{M\in\cM_n}\bigl[ {\rm NonRealTrace}(M,k) \bigr]   \Bigr)
\biggr| \le M n \Lambda_0^k.
$$
\end{proof}

\subsection{The Proof of Lemma~\ref{le_real_eigs_bound}}
\label{su_proof_of_first_conceptual_lemma}

\begin{proof}[Proof of Lemma~\ref{le_real_eigs_bound}]
Consider \eqref{eq_matrix_model_exp} for any fixed $r$.
For any $0\le i\le r-1$ we have
$$
\Ann_{D,L}(S)\bigl(c_i(k)/n^i \bigr) = 
\Ann_{D,L}(S)\bigl(c_i(k)\bigr)/n^i , 
$$
so
Lemmas~\ref{le_c_i} implies that for any fixed,
sufficiently large $D\in\naturals$ there is a function $f_{0,i}(k)$
of growth $\Lambda_0$ for which
$$
\bigl| \Ann_{D,L}(S)\bigl(c_i(k)/n^i \bigr) \bigr|=
\bigl| \Ann_{D,L}(S)\bigl(c_i(k)\bigr) \bigr|/n^i
\le f_{0,i}(k)/n^i \le f_{0,i}(k)
$$
since $i\ge 0$.  Adding the above equation applied to $i=0,\ldots,r-1$
and applying the triangle inequality shows that
$$
\bigl| \Ann_{D,L}(S) \bigl( c_0(k)+\ldots+c_{r-1}/n^{r-1} \bigr) \bigr|
\le F_0(k),
\quad\mbox{where}\quad
F_0(k)=f_{0,0}(k)+\cdots+f_{0,r-1}(k),
$$
is a function of growth $\Lambda_0$.  The term
$O(c_r(k)) n^{-r}$ in
\eqref{eq_matrix_model_exp}
refers to a function of $k,n$ that is
bounded by $n^{-r}$ times a function of growth $\Lambda_1$;
Lemma~\ref{le_ann_c_r} implies the same bound when
$\Ann_{D,L}(S)$ is applied;
hence,
if ${\rm RHS}={\rm RHS}(k,n)$ denotes the
right-hand-side of
\eqref{eq_matrix_model_exp},
\begin{equation}\label{eq_RHS_Ann_bound}
\bigl| \Ann_{D,L}(S) {\rm RHS}(k,n) \bigr| \le F_0(k) + n^{-r} F_1(k),
\end{equation}
where for $i=0,1$, $F_i$ is of growth $\Lambda_i$, and provided that
$1\le k\le K(n)-D(\#L)$ (since applying $\Ann_{D,L}(S)$ requires
\eqref{eq_matrix_model_exp} to hold for $k,k+1,\ldots,k+D(\#L)$).
Set
\begin{equation}\label{eq_define_gs}
g_1(k,n) \eqdef
\EE_{M\in\cM_n}\bigl[ {\rm RealTrace}(M,k) \bigr],\quad
g_2(k,n) \eqdef
\EE_{M\in\cM_n}\bigl[ {\rm NonRealTrace}(M,k) \bigr] \ .
\end{equation}
Lemma~\ref{le_nonreal_eigs_bound} implies that
$$
\bigl| \Ann_{D,L}(S) g_2(k,n)  \bigr| \le n \tilde f(k)
$$
where $\tilde f(k)$ is a function of growth $\Lambda_0$.
Since $g_1(k,n)+g_2(k,n)$ equals the left-hand-side of
\eqref{eq_matrix_model_exp}, we have
\begin{equation}\label{eq_g_one_ann}
\Ann_{D,L}(S) g_1(k,n) = -\Ann_{D,L}(S) g_2(k,n) +
\Ann_{D,L}(S) ({\rm RHS}(k,n)) .
\end{equation} 
Using \eqref{eq_RHS_Ann_bound} and \eqref{eq_le_non_real_trace_bound}
% the fact that $g_2(k,n)$ is bounded
% by $n$ times a function of growth $\Lambda_0$, 
we establish
\eqref{eq_le_real_trace_bound} with $f_0(k)=\tilde f(k)+F_0(k)$ and
$f_1(k)=F_1(k)$.
\end{proof}

\subsection{Proof of Lemma~\ref{le_real_eigs_hp_bound}}
\label{su_proof_of_second_conceptual_lemma}

\begin{proof}[Proof of Lemma~\ref{le_real_eigs_hp_bound}]
Since $D$ is even and $L$ is closed under complex conjugation,
then
$\Ann_{D,L}$ maps real values to non-negative values
(see \eqref{eq_positivity} and the discussion above it), and hence
$$
\Ann_{D,L}(\mu) \ge 0, 
$$
and therefore
$$
\Ann_{D,L}(S)\mu^k = \Ann_{D,L}(\mu)\mu^k \ge 0
$$
if evaluated on an even $k\in\naturals$ and $\mu\in\reals$.
It follows that if $k$ is even, $n$ is fixed, and
$X_n\subset \reals$ is a measurable set,
then
\begin{equation}\label{eq_X_lower_bound}
\left( {\rm inf}_{x\in X_n} \bigl( \Ann_{D,L}(x) x^k \bigr) \right)
\EE_{M\in\cM_n}
\bigl[\#( \Spec(M)\cap X_n)\bigr]
 \le \Ann_{D,L}(S) 
\EE_{M\in\cM_n}\bigl[ {\rm RealTrace}(M,k) \bigr]
\end{equation} 
(and both the above numbers are non-negative reals).

Now take 
$$
X_n = [-\Lambda_1,\Lambda_1] \setminus 
\bigl( B_{\Lambda_0+\epsilon}(0)\cup B_{n^{-\theta}}(L) \bigr) 
$$
and consider any $x\in X_n$.
Then $|x-\ell|\ge n^{-\theta}$ for each
$\ell\in L$, and hence
$$
\Ann_{D,L}(x) \ge n^{-\theta D (\# L) };
$$
in addition $|x|\ge \Lambda_0+\epsilon$ since $x\in X_n$, 
and therefore
$$
\bigl( \Ann_{D,L}(x) x^k \bigr)\ge 
n^{-\theta D (\# L) } (\Lambda_0+\epsilon)^k \ .
$$
Taking the minimum over all $x\in X_n$ and combining this with 
\eqref{eq_X_lower_bound} yields
\eqref{eq_E_n_theta_ineq}.
\end{proof}

\subsection{Proof of Lemma~\ref{le_sidestep_one}}

\begin{proof}[Proof of Lemma~\ref{le_sidestep_one}]
By definition $L_r$, the set of 
larger bases (with respect to $\Lambda_0$)
of the polyexponential parts of $c_0(k),\ldots,c_{r-1}(k)$ 
in an expansion
\eqref{eq_matrix_model_exp}, consists entirely of real numbers;
hence $L_r$ is closed under conjugation.
In view of Lemmas~\ref{le_real_eigs_bound} and 
\ref{le_real_eigs_hp_bound} we have that for sufficiently large even 
$D\in\naturals$,
for any $\theta>0$ and all positive, even $k\le K(n)-D(\#L)$,
\begin{equation}\label{eq_almost_done}
n^{-\theta D (\# L) } (\Lambda_0+\epsilon)^k 
\EE{\rm out}_{\cM_n}\bigl[ B_{\Lambda_0+\epsilon}(0)\cup B_{n^{-\theta}}(L_r) \bigr]
\le  f_0(k)n+f_1(k)n^{-r},
\end{equation} 
where $f_i$ are of growth $\Lambda_i$, for a function $K(n)\gg \log n$.
Since the $f_i$ are of growth $\Lambda_i$,
for every $\delta>0$ we have
$$
n^{-\theta D (\# L) } (\Lambda_0+\epsilon)^k 
\EE{\rm out}_{\cM_n}\bigl[ B_{\Lambda_0+\epsilon}(0)\cup B_{n^{-\theta}}(L_r) \bigr]
\le  (\Lambda_0+\delta)^k n+(\Lambda_1+\delta)^k n^{-r} 
$$
for $n$ large; hence
$$
\EE{\rm out}_{\cM_n}\bigl[ B_{\Lambda_0+\epsilon}(0)\cup B_{n^{-\theta}}(L_r) \bigr]
\le n^{\theta D (\# L) } (\Lambda_0+\epsilon)^{-k}
\bigl( (\Lambda_0+\delta)^k n+(\Lambda_1+\delta)^k n^{-r} \bigr).
$$
Hence, if for some even $k$ we can show that
\begin{equation}\label{eq_last_thing_for_r_one}
n^{\theta D (\# L) } (\Lambda_0+\epsilon)^{-k}
\max\bigl( (\Lambda_0+\delta)^k n,(\Lambda_1+\delta)^k n^{-r} \bigr)
\le n^{-\alpha-\delta'}
\end{equation} 
for some $\delta'>0$,
then it would follow that 
$$
\EE{\rm out}_{\cM_n}\bigl[ B_{\Lambda_0+\epsilon}(0)\cup B_{n^{-\theta}}(L_r) \bigr]
\le n^{-\alpha}
$$
for large $n$, which is just \eqref{eq_sidestep_r_one}.

Fix a
$\kappa\in\reals$---to be specified later---and let $k$ be the nearest
even integer to $\kappa\log n$ 
(which is typical in trace methods for regular graphs), where
$\log$ is with respect to any fixed base (common to all logarithms used).
Taking logarithms in 
\eqref{eq_last_thing_for_r_one} and dividing by $\log n$, it suffices to
choose $\delta,\kappa,\theta,r$ so that
$$
\theta D (\# L) - \kappa \log(\Lambda_0+\epsilon) +
\max\big(1+\kappa \log(\Lambda_0+\delta),-r+\kappa \log(\Lambda_1+\delta)\bigr)
< -\alpha-\delta' \ .
$$
By continuity of the above expression
in $\delta',\delta,\theta$, for a fixed $\kappa,r$ it would be 
possible to find $\delta',\delta,\theta>0$ satisfying the above provided that
the above inequality holds (strictly) for $\delta'=\delta=\theta=0$, i.e., 
that
% $$
% - \kappa \log(\Lambda_0+\epsilon) +
% \max\big(1+\kappa \log\Lambda_0,-r+\kappa \log\Lambda_1)\bigr)
% < -\alpha \ .
% $$
\begin{eqnarray}
\label{eq_r_kappa_first}
- \kappa \log(\Lambda_0+\epsilon) & 
% \max\big(1+\kappa \log\Lambda_0,-r+\kappa \log\Lambda_1)\bigr)
+1+\kappa \log\Lambda_0
& < -\alpha \ ,
\\
\label{eq_r_kappa_second}
- \kappa \log(\Lambda_0+\epsilon) & 
% \max\big(1+\kappa \log\Lambda_0,-r+\kappa \log\Lambda_1)\bigr)
-r+\kappa \log\Lambda_1
& < -\alpha \ .
\end{eqnarray}
To choose such $\kappa,r$, it suffices to first choose $\kappa$ to
satisfy \eqref{eq_r_kappa_first},
which is possible since $\log(\Lambda_0+\epsilon)>\log\Lambda_0$
since $\epsilon>0$; then we choose $r$ to satisfy 
\eqref{eq_r_kappa_second}.
It follows that there exist $\kappa,r$ and $\delta',\delta,\theta>0$ 
that depend only on $\Lambda_0,\Lambda_1,\epsilon,\alpha$ that
satisfy \eqref{eq_last_thing_for_r_one}, and therefore
\eqref{eq_sidestep_r_one} holds for $n$ sufficiently large.

If we set $r_0,\theta_0$ to be those particular values of $r,\theta$
(which depend only on $\Lambda_0,\Lambda_1,\epsilon,\alpha$),
then for any other $r\ge r_0$ and $\theta\le \theta_0$ we 
also have \eqref{eq_sidestep_r_one}, since in this case we have
$$
B_{n^{-\theta_0}}(L_{r_0})   
\subset
B_{n^{-\theta}}(L_r)   
$$
for all $n\in\naturals$.
\end{proof}

% \myDeleteNote{A bunch of stuff deleted on July 7}

We remark that 
\eqref{eq_r_kappa_first} holds iff
\begin{align}
\label{eq_first_ex_kappa}
\kappa & > (\alpha+1)/ \log\left(\frac{\Lambda_0+\epsilon}{\Lambda_0}\right)
\ ,   \\
\label{eq_first_ex_r}
r & > \alpha + \kappa \log\left(\frac{\Lambda_1}{\Lambda_0+\epsilon}\right) \ .
\end{align}
Hence if $\Lambda_0,\Lambda_1$ are fixed,
then $\kappa$ must be at least 
proportional to $\alpha/\epsilon$ as $\alpha\to\infty$ and $\epsilon\to 0$,
and hence the same is true of $r$;
hence we need arbitrarily high order expansions for the model $\cM_n$
as $\alpha\to\infty$ and $\epsilon\to 0$.

\section{Proof of The Sidestepping Lemma (Lemma~\ref{le_sidestep_two})}
\label{se_r1}

In this section we give a variant of Lemma~\ref{le_real_eigs_bound} that
allows us to isolate the $\cM_n$ expected value of eigenvalues near
points outside of $B_{\Lambda_0+\epsilon}$; we then combine this result
with that of \eqref{eq_sidestep_r_one} (in the Exceptional Eigenvalue Bound,
i.e., Lemma~\ref{le_sidestep_one}) to prove
Lemma~\ref{le_sidestep_two}.

\subsection{A Variant of Lemma~\ref{le_real_eigs_bound}}

\begin{lemma}\label{le_real_eigs_bound_with_ell}
Let $\{\cM_n\}_{n\in N}$ be a $(\Lambda_0,\Lambda_1)$-bounded matrix model
that has an order $r$ expansion \eqref{eq_matrix_model_exp} 
with range $K(n)$ (as in Definition~\ref{de_matrix_model}).
Let $L\subset\complex$
be any finite set including all the
bases of $c_0(k),\ldots,c_{r-1}(k)$ in an expansion
\eqref{eq_matrix_model_exp}.  For some $\ell\in L$, let
$\widetilde L= L\setminus\{\ell\}$, and for $i=0,\ldots,r-1$
let $p_i=p_i(k)$ be the polynomial such that $p_i(k)\ell^k$ is the
$\ell$-part of $c_i(k)$.
Then for sufficiently large $\widetilde D\in\naturals$,
for all $k$ with $1\le k\le K(n)-\widetilde D(\# \widetilde L)$ we have
\begin{equation}\label{eq_le_real_trace_bound_r1}
\biggl|
\Ann_{\widetilde D,\widetilde L}(S) 
\Bigl( \EE_{M\in\cM_n}\bigl[ {\rm RealTrace}(M,k) \bigr]  - 
\ell^k\sum_{i=0}^{r-1} p_i(k)/n_i \Bigr)
\biggr|
\le f_0(k)n+f_1(k)n^{-r},
\end{equation}
where $f_0,f_1$ (depending on $\widetilde D,\widetilde L,r$ and the $c_i$,
but not on $n,k$)
are functions of respective growths $\Lambda_0,\Lambda_1$.
\end{lemma}
\begin{proof}
Setting
\begin{equation}\label{eq_widetidle_g_1_def}
\widetilde g_1(k,n)\eqdef
\EE_{M\in\cM_n}\bigl[ {\rm RealTrace}(M,k) \bigr]  - 
\sum_{i=0}^{r-1}p_i(k)\ell^k/n^{-i}
\end{equation} 
and $g_2(k,n)$ as in \eqref{eq_define_gs}, we have
\begin{equation}\label{eq_new_expansion}
\widetilde g_1(k,n)+g_2(k,n)=\sum_{i=0}^{r-1}\frac{\widetilde c_i(k)}{n^i} + 
\frac{O\bigl(c_r(k)\bigr)}{n^r} \ ,
\end{equation} 
where $\widetilde c_i(k)=c_i(k)-p_i(k)\ell^k$; therefore all the
larger bases (with respect to $\Lambda_0$) of the 
$\widetilde c_i(k)$ lie in $\widetilde L$.
Now to both sides of
\eqref{eq_new_expansion}
we apply $\Ann_{\widetilde D,\widetilde L}$, 
subtract $g_2(k,n)$, take absolute values, and apply the triangle
inequality to get
$$
\bigl| \Ann_{\widetilde D,\widetilde L}\widetilde g_1(k,n) \bigr| \le  
\bigl| \Ann_{\widetilde D,\widetilde L}g_2(k,n) \bigr| +
\bigl| \Ann_{\widetilde D,\widetilde L}c_r(k)/n^r \bigr| +
\sum_{i=0}^{r-1} \bigl| \Ann_{\widetilde D,\widetilde L}
\widetilde c_i(k)/n^i \bigr|
$$
Applying
Lemmas~\ref{le_c_i}---\ref{le_nonreal_eigs_bound} (as they were
in the proof of Lemma~\ref{le_real_eigs_bound}), we conclude that
$$
\bigl| \Ann_{\widetilde D,\widetilde L}g_2(k,n) \bigr| \le 
n f_0(k),
\quad
\bigl| \Ann_{\widetilde D,\widetilde L}c_r(k)/n^r \bigr| \le 
n^{-r} f_1(k),
$$
where $f_i(k)$ are functions of growth $\Lambda_i$ ($i=1,2$), and
$$
\bigl| \Ann_{\widetilde D,\widetilde L}
\widetilde c_i(k)/n^i \bigr| \le f_{0,i}(k)/n^i
$$
where $f_{0,i}(k)$ is a function of growth $\Lambda_0$
for $i=0,\ldots,r-1$.
Letting $\tilde f_0$ be $f_0$ plus the sum of the $f_{0,i}$ we have
$$
\bigl| \Ann_{\widetilde D,\widetilde L}\widetilde g_1(k,n) \bigr| \le  
\tilde f_0(k) n + f_1(k) n^{-r},
$$
where $\tilde f_0,f_1$ are of respective growths $\Lambda_0,\Lambda_1$.
This bound and \eqref{eq_widetidle_g_1_def} imply
\eqref{eq_le_real_trace_bound_r1}.
\end{proof}

\subsection{A Variant of Lemma~\ref{le_real_eigs_hp_bound}}

Next we prove a variant of Lemma~\ref{le_real_eigs_hp_bound} that allows
us to relate the term 
$$
\Ann_{\widetilde D,\widetilde L}(S) 
\EE_{M\in\cM_n}\bigl[ {\rm RealTrace}(M,k) \bigr] 
$$
(which implicitly forms part of the left-hand-side of
\eqref{eq_le_real_trace_bound_r1})
to the quantity
$\EE{\rm in}_{\cM_n}[B_{n^{-\theta}}(\ell) ]$,
with notation as in Lemma~\ref{le_real_eigs_bound}.
Let us state this result formally, in a way that will be useful to us.

\begin{lemma}\label{le_control_expectation}
Let $r\in\naturals$, and 
let $\{\cM_n\}_{n\in N}$ be a $(\Lambda_0,\Lambda_1)$-bounded matrix model
that has an order $r$ expansion of range $K(n)$
(\eqref{eq_matrix_model_exp}, Definition~\ref{de_matrix_model}).
Let $\widetilde D\ge 0$ be an integer and $\epsilon,\theta>0$ real numbers;
let $\ell\in\reals$ with $|\ell|>\Lambda_0$ and
$\widetilde L\subset\complex\setminus\{\ell\}$ a finite
set such that $L={\ell}\cup \widetilde L$ lies in
$B_{\Lambda_1}(0)$.
% such that $\ell$ and the elements of $\widetilde L$ are
% each of absolute value greater than $\Lambda_0+\epsilon$.  
Then there is a constant, $C$, independent of
$k$ and $n$, such that 
for large $n\in N$ and all $1\le k\le K(n)-\widetilde D(\#L)$
\begin{align}
\label{eq_ann_real_ell}
& \Big| 
\Big(1+O\bigl(n^{-\theta}\bigr)\Bigr)
\Ann_{\widetilde D,\widetilde L}(\ell) \ell^k 
\EE{\rm in}_{\cM_n}[B_{n^{-\theta}}(\ell) ]
-\Ann_{\widetilde D,\widetilde L}(S) 
\EE_{M\in\cM_n}\bigl[ {\rm RealTrace}(M,k) \bigr] \Bigr| 
\\
\label{eq_stuff_to_bound}
&\le 
C(\Lambda_0+\epsilon)^k n + 
C \Lambda_1^k
\Bigl( n^{1-\theta\widetilde D}+
\EE{\rm out}_{\cM_n}\bigl[B_{\Lambda_0+\epsilon}(0)\cup 
B_{n^{-\theta}}(L)\bigr] 
\Bigr)  
\end{align}
provided that $K(n)\ll n^{\theta}$
(i.e., $K(n)/n^{\theta}\to 0$ as $n\to\infty$).
\end{lemma}
Note that since we only need $K(n) \gg \log n$ in 
Definition~\ref{de_matrix_model} (and in our main
lemmas and Theorem~\ref{th_sidestep}), 
we can always replace $K(n)$ by
a smaller function (e.g., $\max(K(n),(\log n)^2)$)
that is $\gg\log n$ and $\ll n^{\theta}$ (for all $\theta>0$);
hence the assumption that $K(n)\ll n^{\theta}$ is harmless to our
main results.
We remark that
under further assumptions on $\widetilde L$ and 
$\theta$, such as those in
Lemma~\ref{le_real_eigs_bound}, we can make the $\EE{\rm out}$ term in
\eqref{eq_stuff_to_bound}
smaller than $n^{-\alpha}$ for any $\alpha$; however,
the above lemma has no such assumptions.
\begin{proof}
For each square matrix $M$, each $\cR\subset \complex$, and $k\in\naturals$,
let
\begin{equation}\label{eq_real_trace_piece}
{\rm RealTrace}(M,k;\cR) \eqdef
\sum_{\lambda\in \Spec(M)\cap\cR\cap \reals}\lambda^k.
\end{equation}
For each $n$ let
\begin{align*}
S_0 = S_0(n) & = B_{n^{-\theta}}\bigl(\ell\bigr), \\
S_1 = S_1(n) & = B_{n^{-\theta}}\bigl(\widetilde L\bigr),  \\
S_2 = S_2(n) & = B_{\Lambda_0+\epsilon}\bigl(0 \bigr), \\
S_3 = S_3(n) & = \reals\setminus 
\bigl( B_{\Lambda_0+\epsilon}(0)\cup B_{n^{-\theta}}(L) \bigr) ,
% \\
% T_0 = T_0(n) & = \bigl( S_1(n)\cup S_2(n) \cup S_3(n) \bigr) \setminus S_0(n).
\end{align*}
and
$$
g_0(k,n) = {\rm RealTrace}\bigl(M,k; S_0(n) \bigr)
=
\EE{\rm in}_{\cM_n}\bigl[B_{n^{-\theta}}\bigl(\ell\bigr)\bigr].
$$
We may write
\begin{align*}
\EE_{M\in \cM_n}\Bigl[ {\rm RealTrace}(M,k) \Bigr] & -
\EE_{M\in \cM_n}\Bigl[ {\rm RealTrace}\bigl(M,k; S_0(n) \bigr) \Bigr]  
\\
& =
\EE_{M\in \cM_n}\Bigl[ {\rm RealTrace}\bigl(M,k; \reals \setminus
S_0(n) \bigr) \Bigr]  .
\end{align*}
Noting that $\reals$ is contained in the union of the $S_i$, 
applying $\Ann_{\widetilde D,\widetilde L}$ to both sides we have
\begin{equation}\label{eq_g0_ann_realtrace}
\biggl|
\Ann_{\widetilde D,\widetilde L}(S)
\Bigl( g_0(k,n) - \EE_{M\in \cM_n}\bigl[ {\rm RealTrace}(M,k) \bigr] \Bigr)
\biggr|
\le
\sum_{i=1}^3 
E_i M_i,
\end{equation}
where for $i=1,2,3$
$$
E_i = \EE{\rm in}_{\cM_n}[S_i(n)], \quad
M_i = \max_{x\in S_i(n)\cap B_{\Lambda_1}(0)} 
\bigl| \Ann_{\widetilde D,\widetilde L}(S) x^k 
\bigr|.
$$
Let us estimate $E_iM_i$ for $i=1,2,3$.

Since $\widetilde D,\widetilde L$ are fixed, we have
$$
M_2 \le C \Bigl( \, \max_{x\in S_i(n)} |x| \Bigr)^k = C(\Lambda_0+\epsilon)^k,
$$
and so the trivial estimate $E_2\le n$ implies that
\begin{equation}\label{eq_E2M2}
E_2 M_2 \le C n (\Lambda_0+\epsilon)^k.
\end{equation} 

If $x\in B_{n^{-\theta}(\widetilde\ell)}$ for some $\widetilde\ell\in
\widetilde L$, then
$$
\bigl| (x-\widetilde\ell)^{\widetilde D} \bigr| \le n^{-\theta\widetilde D},
$$
and hence 
$$
|\Ann_{\widetilde D,\widetilde L}(x)|
= |x-\widetilde\ell|^{\widetilde D} 
\prod_{\ell'\in\widetilde L\setminus\{\widetilde\ell\}  }
|x-\ell'|^{\widetilde D}
\le n^{-\theta\widetilde D}  (\Lambda_1+|\ell'|)^{\widetilde D} .
$$
In view of \eqref{eq_apply_shift_exponential} we have
$$
|\Ann_{\widetilde D,\widetilde L}(S) x^k| 
= | \Ann_{\widetilde D,\widetilde L}(x) | |x|^k
\le n^{-\theta\widetilde D}  C \Lambda_1^k .
$$
Combining this with the trivial $E_1 \le n$ yields
\begin{equation}\label{eq_E1M1}
E_1 M_1 \le C n^{1-\theta\widetilde D} \Lambda_1^k .
\end{equation} 

Finally
\begin{equation}\label{eq_E3}
E_3 = \EE{\rm out}_{\cM_n}\bigl[B_{\Lambda_0+\epsilon}(0)\cup
B_{n^{-\theta}}(L)\bigr]
\end{equation} 
and 
\begin{equation}\label{eq_M3}
M_3 \le \max_{|x|\le \Lambda_1} 
\bigl( |\Ann_{\widetilde D,\widetilde L}(x)| |x|^k \bigr)
\le C \Lambda_1^k.
\end{equation} 

Combining 
\eqref{eq_E2M2}, \eqref{eq_E1M1}, \eqref{eq_E3}, 
\eqref{eq_M3},
with
\eqref{eq_g0_ann_realtrace} we have
\begin{equation}\label{eq_ann_applied_to_two}
\biggl|
\Ann_{\widetilde D,\widetilde L}(S)
\Bigl( g_0(k,n) - \EE_{M\in \cM_n}\bigl[ {\rm RealTrace}(M,k) \bigr] \Bigr)
\biggr|
\end{equation} 
is bounded by \eqref{eq_stuff_to_bound}.

Since \eqref{eq_ann_applied_to_two} is bounded above by
\eqref{eq_stuff_to_bound}, and since \eqref{eq_ann_applied_to_two}
clearly equals
$$
\biggl|
\Ann_{\widetilde D,\widetilde L}(S)
 g_0(k,n) - 
\Ann_{\widetilde D,\widetilde L}(S)
\EE_{M\in \cM_n}\bigl[ {\rm RealTrace}(M,k) \bigr] 
\biggr| ,
$$
to finish the proof of Lemma~\ref{le_control_expectation} it suffices to 
show that
\begin{equation}\label{eq_last_part}
\Ann_{\widetilde D,\widetilde L}(S) g_0(k,n) 
=
\Big(1+O\bigl(n^{-\theta}\bigr)\Bigr)
\Ann_{\widetilde D,\widetilde L}(\ell) \ell^k \EE{\rm in}_{\cM_n}\bigl[B_{n^{-\theta}}(\ell)\bigr] 
\end{equation} 
for fixed $\widetilde D,\widetilde L,\ell$, and large $n,k$ with
$k\ll n^{\theta}$.
Let us prove this.

Since $\ell\in\reals$,
% and $|\ell|>\Lambda_0$, for sufficiently large
% $n$ we have $|\ell|+n^{-\theta}>\Lambda_0$ and 
the set
$B_{n^{-\theta}}(\ell)\cap\reals$ is just the closed interval
$I_n=[\ell-n^{-\theta},\ell+n^{-\theta}]$; so setting
$$
h(x) \eqdef \Ann_{\widetilde D,\widetilde L}(x)x^k,
$$
we have that for any $\lambda\in I_n$ there is a $\xi\in I_n$ for which
$$
h(\lambda)-h(\ell) = (\lambda-\ell) h'(\xi).
$$
The product rule shows that
$$
|h'(x)| \le C |\xi|^k \le C(|\ell|+n^{-\theta})^k\le C'|\ell|^k
$$
provided that $k\ll n^{\theta}$.
Hence
\begin{equation}\label{eq_h_mean_value_est}
|h(\lambda)-h(\ell) | \le |\lambda-\ell| C'|\ell|^k
\le C' n^{-\theta} |\ell|^k.
\end{equation} 
Since $\ell\notin\widetilde L$,
the quantity
$$
A = \Ann_{\widetilde D,\widetilde L}(\ell) \ne 0
$$
(and is, of course, independent of $k,n$); since
$h(\ell)=A\ell^k$ (by \eqref{eq_apply_shift_exponential}), 
\eqref{eq_h_mean_value_est} implies
$$
h(\lambda)=
h(\ell) + O(n^{-\theta})|\ell^k|
=A\ell^k + O(n^{-\theta})|\ell^k| 
=A\ell^k \Bigl( 1 + O\bigl(n^{-\theta}\bigr) \Bigr),
$$
which holds for all $\lambda\in I_n=B_{n^{-\theta}}(\ell)\cap \reals$.
It follows that for any $M\in\cM_n$ we have
$$
\Ann_{\widetilde D,\widetilde L}(S) % g_0(k,n) 
{\rm RealTrace}\Bigl(M,k; B_{n^{-\theta}}(\ell) \Bigr)
= \bigl( \#\Spec(M)\cap B_{n^{-\theta}}(\ell) \bigr)
A \ell^k \Bigl( 1 + O\bigl(n^{-\theta} \bigr) \Bigr)  .
$$
Taking expected values yields
\begin{align*}
\Ann_{\widetilde D,\widetilde L}(S) g_0(k,n) 
& =
\EE{\rm in}_{\cM_n}\bigl[B_{n^{-\theta}}(\ell)\bigr]
A \ell^k \Bigl( 1 + O\bigl(n^{-\theta} \bigr)\Bigr) 
\\
& =
\EE{\rm in}_{\cM_n}\bigl[B_{n^{-\theta}}(\ell)\bigr]
\Ann_{\widetilde D,\widetilde L}(\ell) 
\ell^k \Bigl( 1 + O\bigl(n^{-\theta} \bigr) \Bigr)
\end{align*}
which proves \eqref{eq_last_part};
by the discussion in the paragraph containing \eqref{eq_last_part},
this completes the proof.
\end{proof}

\subsection{A Shift Computation}

Here is a simple lemma regarding the shift operation which we will need
regarding the left-hand-side of
\eqref{eq_le_real_trace_bound_r1}.

\begin{lemma}\label{le_shift_computation}
Let $p=p(k)$ be a non-zero polynomial with
complex coefficients, let $\ell\in\complex$, and let
$\widetilde L\subset \complex\setminus\{\ell\}$ be a finite set.
Then for each $\widetilde D\in\naturals$ there is a polynomial
$\widetilde p(k)$ of the same degree as $p(k)$ such that
$$
\Ann_{\widetilde D,\widetilde L}(S)
\bigl( p(k)\ell^k \bigr) = \widetilde p(k) \ell^k;
$$
more precisely, if $p(k)=c_0 + c_1 k + \cdots + c_t k^t$ with $c_t\ne 0$, then
we have
$$
\widetilde p(k) = \widetilde c_0 + \widetilde c_1 k +
\cdots + \widetilde c_t k^t
$$
where $\widetilde c_t\ne 0$ and is given by
$$
\widetilde c_t = c_t \Ann_{\widetilde D,\widetilde L}(\ell) \ne 0.
$$
\end{lemma}
\begin{proof}
For any $\widetilde\ell\ne \ell$, we see that
$$
(S-\widetilde \ell) \bigl( k^t\ell^k \bigr) =
(k+1)^t \ell^{k+1} - k^t \widetilde\ell \ell^k 
= \ell^k \bigl( k^t (\ell-\widetilde\ell) + \cdots \bigr)
$$
where the $\cdots$ is a polynomial in $k$ of degree less than $t$.
It follows that 
$$
(S-\widetilde \ell) \bigl( p(k)\ell^k \bigr) = \ell^k 
\bigl( c_t k^t (\ell-\widetilde\ell) + \cdots \bigr)
$$
where the $\cdots$ is a polynomial in $k$ of degree less than $t$;
this
proves the claim for $\#L=1$ (i.e., $\widetilde L=\{\widetilde\ell\}$)
and $\widetilde D=1$.
Repeatedly applying $S-\widetilde\ell$ as above proves the
case $\widetilde L=\{\widetilde\ell\}$ and all $\widetilde D\in\naturals$.
For general $\widetilde D$ and $\widetilde L$, 
we write $\widetilde L=\{\ell_1,\ldots,\ell_m\}$, and write
$$
\Ann_{\widetilde D,\widetilde L}(S) \bigl(p(k)\ell^k\bigr)
=
\bigl(S-\ell_1\bigr)^{\widetilde D}\ldots
\bigl(S-\ell_m\bigr)^{\widetilde D} \bigl(p(k)\ell^k\bigr).
$$
We then first apply $(S-\ell_m)^{\widetilde D}$ to $ p(k)\ell^k$,
then apply $(S-\ell_{m-1})^{\widetilde D}$ to the result, etc.,
and the lemma follows (more formally, the lemma follows by induction
on $m=\#\widetilde L$).
%
% \Bigl( \Ann_{\widetilde D,L_1}(S) \bigl(p(k)\ell^k\bigr) \Bigr) ,
% $$
% 
% the lemma follows easily
% by induction on $\#\widetilde L$: we have already established the base case.
% For the inductive step we write
% $\widetilde L=L_1\cup\{\ell_1\}$ with $\#L_1<\#\widetilde L$;  the inductive
% hypthosis is that the lemma holds for $L_1$, and all $D$;
% we then write 
% $$
% \Ann_{\widetilde D,\widetilde L}(S) \bigl(p(k)\ell^k\bigr)
% =
% \bigl(S-\ell_1\bigr)^{\widetilde D} 
% \Bigl( \Ann_{\widetilde D,L_1}(S) \bigl(p(k)\ell^k\bigr) \Bigr) ,
% $$
% and apply the inductive hypothesis to
% $$
% \Ann_{\widetilde D,L_1}(S) \bigl(p(k)\ell^k\bigr) .
% $$
% jjjjjjj 
\end{proof}

\subsection{Proof of Lemma~\ref{le_sidestep_two}}

\begin{proof}[Proof of Lemma~\ref{le_sidestep_two}]
Let $\widetilde\epsilon=\epsilon/3$, and let $\kappa_0$ be given by
$$
\kappa_0 = \frac{j+2}{\log\Bigl( (\Lambda_0+2\widetilde\epsilon) /
(\Lambda_0+\widetilde\epsilon) \Bigr)}
$$
where the logarithm, and all logarithms below,
are taken relative to a fixed base [this choice of base does not affect
$r_1,\theta_1$ below and merely scales $\kappa_0$].
Clearly $\kappa_0$ is positive
and can be equivalently be described by the equation
\begin{equation}\label{eq_side_choose_kappa}
\kappa_0 \log(\Lambda_0+2\widetilde\epsilon) - j -2 = 
\kappa_0 \log(\Lambda_0+\widetilde\epsilon) .
\end{equation}
Let $\kappa$ be any real number with
$$
\kappa \ge \kappa_0;
$$
our proof will require us to do the present computation 
twice, with
two different values of $\kappa$; this is why we don't choose
one specific value of $\kappa$.
Then let
\begin{align}
\label{eq_def_widetilde_alpha}
\widetilde\alpha & =
j+1+\kappa \log(\Lambda_1) - \kappa  \log(\Lambda_0+2\widetilde\epsilon) \\
\label{eq_def_widetilde_r}
\widetilde r & = j+1 + 
\lceil \kappa \log(\Lambda_1+\widetilde\epsilon)
- \kappa \log(\Lambda_0+2\widetilde\epsilon) \rceil
\end{align}
where $\lceil\ \rceil$ denotes the ceiling function (the least upper bound
that is an integer).
Since $\Lambda_1\ge 3\widetilde\epsilon + \Lambda_0$, we have
$\widetilde\alpha, \widetilde r>0$.

We then set
\begin{align}
\label{eq_def_r_1}
r_1 &= 
\max\bigl( \widetilde r,
r_0(\Lambda_0,\Lambda_1,\widetilde\epsilon,\widetilde\alpha) \bigr) \\
\label{eq_def_theta_1}
\theta_1 &= \theta_0(\Lambda_0,\Lambda_1,\widetilde\epsilon,\widetilde\alpha) ,
\end{align}
with $r_0,\theta_0$ being the functions in
Lemma~\ref{le_sidestep_one}.
For now we have that $r_1,\theta_1$ depend on $\kappa$ as well as
$\Lambda_0,\Lambda_1,j,\epsilon$
(since $\widetilde r,\widetilde\alpha$ depend on $\kappa$); we will ultimately
take $\kappa=\kappa_0$ to produce the $r_1,\theta_1$ in
Lemma~\ref{le_sidestep_two}.

Next let $\theta\in\reals$ 
and $r\in\naturals$ satisfy
$$
0<\theta \le \theta_1(\Lambda_0,\Lambda_1,j,\epsilon,\kappa),\quad
r \ge r_1(\Lambda_0,\Lambda_1,j,\epsilon,\kappa).
$$
Let us prove \eqref{eq_most_useful}--\eqref{eq_q_ell_j_formula}.

Since
$$
r\ge r_1 \ge r_0(\Lambda_0,\Lambda_1,\widetilde\epsilon,\widetilde\alpha), 
\quad
\theta\le \theta_1 
\le \theta_0(\Lambda_1,\Lambda_0,\widetilde\epsilon,\widetilde\alpha) ,
$$
Lemma~\ref{le_sidestep_one} implies that
\begin{equation}\label{eq_application_exception}
\EE{\rm out}_{\cM_n}
\bigl[B_{\Lambda_0+\widetilde\epsilon}(0)\cup B_{n^{-\theta_1}}(L_{r_1})\bigr]
\le n^{-\widetilde\alpha} \le n^{-j-1}
\end{equation} 
for $n$ sufficiently large; since $\theta\le \theta_1$, 
$r\ge r_1$, and $\widetilde\epsilon<\epsilon$, we have
$$
B_{\Lambda_0+\widetilde\epsilon}(0)\cup B_{n^{-\theta_1}}(L_{r_1})
\subset
B_{\Lambda_0+\epsilon}(0)\cup B_{n^{-\theta}}(L_r)
$$
and hence
\begin{equation}\label{eq_EXE_tildes}
\EE{\rm out}_{\cM_n}
\bigl[B_{\Lambda_0+\epsilon}(0)\cup B_{n^{-\theta}}(L_r)\bigr]
\le n^{-j-1}
\end{equation}
for $n$ sufficiently large.

Note that $r_1\ge \widetilde r$ by \eqref{eq_def_r_1},
and $\widetilde r\ge j+1$ by \eqref{eq_def_widetilde_r},
and hence 
\begin{equation}\label{eq_r_greater_j}
r\ge r_1\ge j+1, \quad\mbox{and}\quad
L_{j+1}\subset L_r .
\end{equation}

Now let $\ell\in L_r$ with $|\ell|\ge \Lambda_0+2\widetilde\epsilon$,
and set
$$
\widetilde L = L_r \setminus \{\ell\}.
$$
Let $\widetilde D\in\naturals$ be sufficiently large so that
Lemma~\ref{le_real_eigs_bound_with_ell} holds and that
\begin{equation}\label{eq_side_choose_widetilde_D}
\kappa \log(\Lambda_0+2\widetilde\epsilon) - j -1 \ge \kappa \log(\Lambda_1)+1-
\theta \widetilde D \ .
\end{equation}
Lemma~\ref{le_real_eigs_bound_with_ell} implies that
for some function $K(n)$ with $K(n)\gg \log n$ we have that
for all $k\in\naturals$
with $1\le k\le K(n)-\widetilde D(\#\widetilde L)$ we have
\begin{equation}\label{eq_apply_first_lemma}
\biggl|
\Ann_{\widetilde D,\widetilde L}(S)
\Bigl( \EE_{M\in\cM_n}\bigl[ {\rm RealTrace}(M,k) \bigr]  -
\sum_{i=0}^{r-1} \ell^k q_{\ell,i}(k)n^{-i} \Bigr)
\biggr| \le f_0(k)n+f_1(k)n^{-r} ,
\end{equation}
% is bounded by a function of $n$ that is
% \begin{equation}\label{eq_bound_of_ann_interesting}
% f_0(k)n+f_1(k)n^{-r} 
% \end{equation}
where $f_0,f_1$ are functions of respective growths $\Lambda_0,\Lambda_1$.

Now for $n$ such that $\cM_n$ is defined, let $k(n)$ be any function
with $k(n)-\kappa\log n=O(1)$.
Let us prove that for $k=k(n)$, as 
$n\to\infty$ (over $n$ where $\cM_n$ is defined) 
we have
\begin{align}
\label{eq_f_0_k_of_n}
\bigl| f_0\bigl( k(n) \bigr) n \bigr| 
& = o \bigl( |\ell|^{k(n)}n^{-j} \bigr) \quad\mbox{and} \\
\label{eq_f_1_n_r}
\bigl| f_1\bigl( k(n) \bigr) n^{-r}  \bigr|
& = o \bigl( |\ell|^{k(n)}n^{-j} \bigr)  .
\end{align}
Since $f_0$ is of growth $\Lambda_0$, we have
\begin{equation}\label{eq_f_0_intermediate}
f_0\bigl(  k(n) \bigr) n  = O(1) ( \Lambda_0+\widetilde\epsilon )^{k(n)} n
\le O(1) ( \Lambda_0+\widetilde\epsilon )^{\kappa\log n+O(1)} n.
\end{equation} 
Since $\kappa\ge\kappa_0$, we have that 
\eqref{eq_side_choose_kappa} also holds with $\kappa$ replacing
$\kappa_0$, and hence multiplying the equation by $\log n$ we have
\begin{equation}\label{eq_Lambda_0_widetidle_eps}
( \Lambda_0+\widetilde\epsilon )^{\kappa\log n} n
\le ( \Lambda_0+2\widetilde\epsilon )^{\kappa\log n} n^{-j-2}
\le |\ell|^{\kappa\log n} n^{-j-2} \le |\ell|^{k(n)+O(1)}n^{-j-2} ;
\end{equation} 
combining this with \eqref{eq_f_0_intermediate} yields
$$
f_0(\bigl( k(n) \bigr) n  = |\ell|^{k(n)}n^{-j}O(1/n) 
$$
which establishes \eqref{eq_f_0_k_of_n}.
The proof of \eqref{eq_f_1_n_r} is similar: in view of
\eqref{eq_def_widetilde_r} and $r\ge \widetilde r$ we have
\begin{align*}
n^{-r} \le n^{-\widetilde r}
& \le n^{-j-1} n^{-\kappa\log(\Lambda_1+\widetilde\epsilon)}
n^{\kappa\log(\Lambda_0+2\widetilde\epsilon)} \\
& = n^{-j-1} O(1) (\Lambda_1+\widetilde\epsilon)^{-k(n)}
(\Lambda_0+2\widetilde\epsilon)^{k(n)} 
\end{align*}
where the last equality holds since $k(n)=\kappa\log n +O(1)$;
since $f_1(k)\le (\Lambda_1+\widetilde\epsilon)^{k(n)}$ for large
$n$, we have
$$
f_1\bigl( k(n) \bigr) n^{-r} = O(1/n)n^{-j}
n^{\kappa\log(\Lambda_0+2\widetilde\epsilon)}
=  |\ell|^{k(n)}n^{-j}O(1/n)
$$
which proves \eqref{eq_f_1_n_r}.

Combining \eqref{eq_apply_first_lemma}--\eqref{eq_f_1_n_r} we have
\begin{equation}\label{eq_mess_versus_ell_k_of_n}
\biggl|
\Ann_{\widetilde D,\widetilde L}(S)
\Bigl( \EE_{M\in\cM_n}\bigl[ {\rm RealTrace}(M,k) \bigr]  -
\sum_{i=0}^{r-1} \ell^k q_{\ell,i}(k)n^{-i} \Bigr)
\biggr| 
\le o \bigl( |\ell|^{k(n)}n^{-j} \bigr)
\end{equation}
provided that $1\le k(n)\le K(n)-\widetilde D(\#\widetilde L)$.
Since $k(n)=\kappa\log n + O(1)$, we have that 
$1\le k(n)\le K(n)-\widetilde D(\#\widetilde L)$ for
sufficiently large $n$.
% \eqref{eq_apply_first_lemma} is bounded by
% $$
% f_0\bigl( k(n) \bigr) n + f_1\bigl( k(n) \bigr) n^{-r}
% = o \bigl( |\ell|^{k(n)}n^{-j} \bigr)  ;
% $$
% therefore \eqref{eq_apply_first_lemma} is bounded as
% $o(|\ell|^{k(n)}n^{-j})$.
With this in mind, we now finish the proof
of Lemma~\ref{le_sidestep_two}
by estimating the left-hand-side of
\eqref{eq_mess_versus_ell_k_of_n}
up to terms of size $o(|\ell|^{k(n)}n^{-j})$.

Applying Lemma~\ref{le_shift_computation} we have 
$$
\Ann_{\widetilde D,\widetilde L}(S)
\ell^{k(n)}\sum_{i=0}^{r-1} q_{\ell,i}\bigl(k(n)\bigr)/n^i =
\ell^{k(n)}\sum_{i=0}^{r-1} Q_i\bigl(k(n)\bigr)/n^i
$$
where $Q_i=Q_i(k)$ is the (unique) polynomial such that
$$
\Ann_{\widetilde D,\widetilde L}(S)
\bigl( \ell^k q_{\ell,i}(k) \bigr)
=
\ell^k Q_i(k);
$$
hence $Q_i(k)$ is of the same degree as $q_{\ell,i}(k)$ and the
leading coefficient of $Q_i(k)$ is 
$\Ann_{\widetilde D,\widetilde L}(\ell)$ times
that of $q_{\ell,i}(k)$.  Since $L_j=\emptyset$ we have
$q_{\ell,i}(k)=0$ for all $i<j$; it follows that
$$
\Ann_{\widetilde D,\widetilde L}(S)
\ell^{k(n)}\sum_{i=0}^{r-1} q_{\ell,i}\bigl(k(n)\bigr)/n^i =
\ell^{k(n)} Q_j(k(n)) n^{-j} + |\ell|^{k(n)} n^{-j} 
O\Bigl( \bigl(k(n)\bigr)^m / n\Bigr)
$$
where $m$ is the degree of $q_{\ell,j+1}(k)$.  Since $k(n)=O(\log n)$
we have
\begin{equation}\label{eq_ann_on_series}
\Ann_{\widetilde D,\widetilde L}(S)
\ell^{k(n)}\sum_{i=0}^{r-1} q_{\ell,i}(k(n))/n^i =
\ell^{k(n)} Q_j(k(n)) n^{-j} + |\ell|^{k(n)} n^{-j} o(1)
\end{equation} 

Next consider
\begin{equation}\label{eq_ann_on_expected}
\Ann_{\widetilde D,\widetilde L}(S)
\EE_{M\in\cM_n}\bigl[ {\rm RealTrace}(M,k(n)) \bigr].
\end{equation} 
Applying Lemma~\ref{le_control_expectation} with 
$\widetilde\epsilon$ replacing $\epsilon$, we have \eqref{eq_ann_on_expected}
equals 
\begin{equation}\label{eq_leftover_from_near_ell}
\Big(1+O\bigl(n^{-\theta}\bigr)\Bigr)
\Ann_{\widetilde D,\widetilde L}(\ell) \ell^{k(n)}
\EE{\rm in}_{\cM_n}[B_{n^{-\theta}}(\ell) ] + O(1){\rm Term}(n)
\end{equation} 
where 
$$
{\rm Term}(n) = {\rm Term}_1(n) + {\rm Term}_2(n)
$$
where
$$
{\rm Term}_1(n) =
(\Lambda_0+\widetilde\epsilon)^{k(n)} n,\quad
{\rm Term}_2(n) =
\Lambda_1^{k(n)}
\Bigl( n^{1-\theta\widetilde D}+
\EE{\rm out}_{\cM_n}\bigl[B_{\Lambda_0+\widetilde\epsilon}(0)\cup
B_{n^{-\theta}}(\widetilde L)\bigr]
\Bigr) ;
$$
let us show that ${\rm Term}_1(n),{\rm Term}_2(n)$ are both
$=o(|\ell|^{k(n)}n^{-j})$:
we have already shown
$$
(\Lambda_0+\widetilde\epsilon)^{k(n)} n  = o(|\ell|^{k(n)}n^{-j})
$$
in \eqref{eq_Lambda_0_widetidle_eps}, so  
${\rm Term}_1(n)=o(|\ell|^{k(n)}n^{-j})$.
Also
\begin{align*}
\EE{\rm out}_{\cM_n}\bigl[
B_{\Lambda_0+\widetilde\epsilon}(0)\cup B_{n^{-\theta}}(L_r)
\bigr] 
& \le
\EE{\rm out}_{\cM_n}\bigl[
B_{\Lambda_0+\widetilde\epsilon}(0)\cup B_{n^{-\theta}}(L_{r_1})
\bigr] 
\le n^{-\widetilde\alpha}  \\
& = n^{-j-1-\kappa\log(\Lambda_1)+\kappa\log(\Lambda_0+2\widetilde\epsilon)}
\end{align*}
by
\eqref{eq_application_exception}.
Furthermore \eqref{eq_side_choose_widetilde_D} implies that
$$
1-\theta\widetilde D \le \kappa \log(\Lambda_0+2\widetilde\epsilon) - j -1
- \kappa \log(\Lambda_1) 
$$
and therefore
$$
n^{1-\theta \widetilde D}  \le 
n^{-j-1-\kappa\log(\Lambda_1)+\kappa\log(\Lambda_0+2\widetilde\epsilon)} .
$$
It follows that 
$$
{\rm Term}_2(n) \le
\Lambda_1^{k(n)}
n^{-j-1-\kappa\log(\Lambda_1)+\kappa\log(\Lambda_0+2\widetilde\epsilon)}
$$
which---using $k(n)=\kappa\log n + O(1)$ (as before)---is bounded by
$$
O(1/n)n^{-j} (\Lambda_0+2\widetilde\epsilon)^{k(n)}
=
o\bigl(|\ell|^{k(n)} n^{-j} \bigr).
$$

Hence ${\rm Term}(n)=o(|\ell|^k n^{-j})$; combining this with
\eqref{eq_ann_on_series},
\eqref{eq_ann_on_expected}, 
\eqref{eq_leftover_from_near_ell}, and 
\eqref{eq_mess_versus_ell_k_of_n} we have that
$$
\Big(1+O\bigl(n^{-\theta}\bigr)\Bigr)
\Ann_{\widetilde D,\widetilde L}(\ell) \ell^k \EE{\rm in}_{\cM_n}\bigl[B_{n^{-\theta}}(\ell)\bigr]
-
\Ann_{\widetilde D,\widetilde L}(S)
\bigl( \ell^k q_{\ell,j}(k)n^{-j} \bigr) = o\bigl(|\ell|^k n^{-j}\bigr) \ .
$$
% Dividing by $\ell^k n^{-j}\Ann_{\widetilde D,\widetilde L}(\ell)$ and applying
Dividing by $\ell^k n^{-j}$ and applying
Lemma~\ref{le_shift_computation} we have that
$$
\Ann_{\widetilde D,\widetilde L}(\ell)
\Big(1+O\bigl(n^{-\theta}\bigr)\Bigr) 
\EE{\rm in}_{\cM_n}\bigl[B_{n^{-\theta}}(\ell)\bigr] n^j 
-
Q_j(k(n)) = o(1).
$$
Dividing by the function that is $1+O(n^{-\theta})$ above, we have
$$
\Ann_{\widetilde D,\widetilde L}(\ell)
\EE{\rm in}_{\cM_n}\bigl[B_{n^{-\theta}}(\ell)\bigr] n^j  
= Q_j(k(n))  (1+o(1)) 
= Q_j(\kappa\log n) (1+o(1))
$$
where the last equality uses the fact that $k(n)=\kappa\log n+O(1)$.

Now we claim that $Q_j(k)$ must be a constant, rather than a polynomial
of degree one or greater. 
Indeed, let $\kappa'\ge\kappa_0$ with $\kappa'\ne\kappa$.  
Repeating all the above with $\kappa'$ replacing $\kappa$ we have
$$
\Ann_{\widetilde D,\widetilde L}(\ell)
\EE{\rm in}_{\cM_n}\bigl[B_{n^{-\theta}}(\ell)\bigr] n^j  
= Q_j(\kappa\log n) (1+o(1))
 = Q_j(\kappa'\log n) (1+o(1))
$$
as $n\to\infty$; the equality
$$
Q_j(\kappa\log n) (1+o(1))
 = Q_j(\kappa'\log n) (1+o(1))
$$
as $n\to\infty$, and $\kappa'\ne \kappa$,
implies that the polynomial $Q_j$ must be a constant.

Since $Q_j=Q_j(k)$ is a constant, the limit
$$
\lim_{n\to\infty} \Bigl(
\Ann_{\widetilde D,\widetilde L}(\ell)
\EE{\rm in}_{\cM_n}\bigl[B_{n^{-\theta}}(\ell)\bigr]
n^j
\Bigr) 
$$
exists and equals the constant $Q_j$.
But Lemma~\ref{le_shift_computation} implies that $q_{\ell,j}=q_{\ell,j}(k)$
is also a constant and 
$$
Q_j = \Ann_{\widetilde D,\widetilde L}(\ell) q_{\ell,j}.
$$
It follows that
\begin{equation}\label{eq_limit_smallest_coefficient}
\lim_{n\to\infty} \Bigl(
\EE{\rm in}_{\cM_n}\bigl[B_{n^{-\theta}}(\ell)\bigr]
n^j
\Bigr)  = q_{\ell,j}.
\end{equation} 
This proves
\eqref{eq_C_ell_alternate_formula} and
\eqref{eq_q_ell_j_formula}.

Next let
$$
L_r' = \{ \ell\in L_r \ | \ |\ell| \ge \Lambda_0+2\epsilon/3 \}.
$$
For sufficiently large $n$---and more precisely for 
$n^{-\theta}<\epsilon/3$ (here we use $\theta>0$)---we have
$$
B_{\Lambda_0+\epsilon}(0)\cup B_{n^{-\theta}}(L_r)
=
B_{\Lambda_0+\epsilon}(0)\cup B_{n^{-\theta}}(L_{j+1})
\cup B_{n^{-\theta}}(L_r'\setminus L_{j+1}).
$$
It follows that
$$
\EE{\rm out}_{\cM_n}\bigl[
B_{\Lambda_0+\epsilon}(0)\cup B_{n^{-\theta}}(L_{j+1})
\bigr]
\le
\EE{\rm out}_{\cM_n}\bigl[
B_{\Lambda_0+\epsilon}(0)\cup B_{n^{-\theta}}(L_r)
\bigr]
+
\EE{\rm in}_{\cM_n}\bigl[
B_{n^{-\theta}}(L_r'\setminus L_{j+1})
\bigr].
$$
But for all $\ell\in L_r'\setminus L_{j+1}$ we have $q_{\ell,j}=0$
since $\ell\notin L_{j+1}$.  
Applying
\eqref{eq_limit_smallest_coefficient} to all
$\ell\in L_r'\setminus L_{j+1}$ and summing over such $\ell$ we have
$$
\EE{\rm out}_{\cM_n}\bigl[
B_{\Lambda_0+\epsilon}(0)\cup B_{n^{-\theta}}(L_{j+1})
\bigr]
\le
\EE{\rm out}_{\cM_n}\bigl[
B_{\Lambda_0+\epsilon}(0)\cup B_{n^{-\theta}}(L_r)
\bigr]
+ o(n^{-j}).
$$
Combining this with
\eqref{eq_EXE_tildes} yields
\eqref{eq_exception_for_side_stepping}.

Similarly we have that if 
$$
L_{j+1}' = \{ \ell\in L_{j+1} \ | \ |\ell| \ge \Lambda_0+2\epsilon/3 \},
$$
then for sufficiently large $n$ 
\begin{align*}
\EE{\rm out}_{\cM_n}\bigl[
B_{\Lambda_0+\epsilon}(0)
\bigr] 
& \le
\EE{\rm out}_{\cM_n}\bigl[
B_{\Lambda_0+\epsilon}(0)\cup B_{n^{-\theta}}(L_{j+1})
\bigr]
+
\EE{\rm in}_{\cM_n}\bigl[
B_{n^{-\theta}}(L_{j+1})
\bigr]
\\
& \le 
\EE{\rm out}_{\cM_n}\bigl[
B_{\Lambda_0+\epsilon}(0)\cup B_{n^{-\theta}}(L_{j+1})
\bigr]
+
O(n^{-j}) ,
\end{align*}
which proves \eqref{eq_most_useful}.
\end{proof}

\section{Proof of the Sidestepping Theorem}
\label{se_side_theorem}

\begin{proof}[Proof of Theorem~\ref{se_side_theorem}]
Using induction on $i$,
the formula \eqref{eq_c_i_limit_formula} shows that 
for any
$i\in\integers_{\ge 0}$, the function $c_i=c_i(k)$ appearing in
the expansions \eqref{eq_matrix_model_exp} are independent of
$r$ over all $r\ge i+1$ (for $r\le i-1$, the function $c_i$ does not
exist in \eqref{eq_matrix_model_exp}, and for $r=i$ the term
$O(c_i(k))/n^i$ does not uniquely determine $c_i$).

If $p_j(k)$, the polyexponential part of $c_j(k)$ with respect to
$\Lambda_0$, vanishes for all $j$, 
then according to Lemma~\ref{le_sidestep_two}, \eqref{eq_most_useful},
$$
\EE{\rm out}_{\cM_n}[B_{\Lambda_0+\epsilon}(0)] = O(n^{-j})
$$
for all $j$.

Otherwise let $j$ be the smallest integer such that $p_j(k)\ne 0$.
Then $p_j(k)$ has a finite number of bases, so there exists an
$\epsilon_0>0$ such that (1) all of these bases have absolute value
at least
$\Lambda_0+2\epsilon_0/3$, and (2) $\Lambda_0+\epsilon_0\le \Lambda_1$.
According to Lemma~\ref{le_sidestep_two}, for any $\epsilon>0$
with $\epsilon\le\epsilon_0$, the existence of an
expansion of order $r_1=r_1(\Lambda_0,\Lambda_1,j,\epsilon)$
for $\cM_n$ (which exist for any order) implies
\eqref{eq_exception_for_side_stepping} and that
$$
p_j(k) = \sum_{\ell\in L_{j+1}} C_\ell \ell^k
$$
where the $C_\ell>0$ are constants given by
\eqref{eq_C_ell_alternate_formula}
for $\theta>0$ satisfying 
$\theta\le\theta_1(\Lambda_0,\Lambda_1,j,\epsilon)$.
This proves
\eqref{eq_mostly_near_Lambda_0_or_Ls}---\eqref{eq_thm_C_ell_as_limit},
provided that $\epsilon\le\epsilon_0$ and $\theta\le\theta_1$;
of course, if $\epsilon\ge\epsilon_0$, then
\eqref{eq_mostly_near_Lambda_0_or_Ls} for $\epsilon=\epsilon_0$
also implies this for any larger value of $\epsilon$.
\end{proof}

% Fix a $j\in\integers_{\ge 0}$ and an $r\ge j$, so that
% $c_0,\ldots,c_{j-1}$ are uniquely determined by \eqref{eq_c_i_limit_formula}.
% If $L_j$, the larger bases of $c_0,\ldots,c_{j-1}$, is the empty set,
% then according to Lemma~\ref{le_sidestep_two}, \eqref{eq_most_useful},
% $$
% \EE{\rm out}_{\cM_n}[B_{\Lambda_0+\epsilon}(0)] = O(n^{-j})
% $$
% (since $\cM_n$ has expansions to arbitrarily high order).
% If $L_j$ is empty for all $j$, then the above holds for all $j$.
% Otherwise let $j\in\integers
% This either holds for all $j$, or else $L_j$ is nonempty for some
% $j\in\naturals$; in the latter case, consider the smallest such $j$.
% Then, consider an order $r$ expansion \eqref{eq_matrix_model_exp}
% for $r=j$: $c_0(k),\ldots,c_{j-1}(k)$ are uniquely determined 
% by \eqref{eq_c_i_limit_formula}, and is a approximate
%%  polyexponential with error growth $\Lambda_0$; the set of larger
% bases (with respect to $\Lambda_0$) of $c_{j-1}(k)$ is the finite
% set $L_j$.  Therefore there exists an $\epsilon_0$ such that
% all 

%    Bibliography styles amsplain or harvard are also acceptable.
\providecommand{\bysame}{\leavevmode\hbox to3em{\hrulefill}\thinspace}
\providecommand{\MR}{\relax\ifhmode\unskip\space\fi MR }
\providecommand{\MRhref}[2]{%
  \href{http://www.ams.org/mathscinet-getitem?mr=#1}{#2}
}
\providecommand{\href}[2]{#2}

\end{document}